\documentclass{amsart}
\usepackage{}
\usepackage{bbm}
\usepackage{amsfonts}
\usepackage{amsmath, amsthm, amsfonts, amssymb}
\usepackage{mathrsfs}
\usepackage{amsmath, color}
\DeclareMathOperator*{\supp}{supp}

\def\supp{{\rm{\ supp\ }}}

\def\supp{{\rm supp}}

\newtheorem{thm}{Theorem}[section]
\newtheorem{lem}{Lemma}[section]

\begin{document}
\title[Bochner-Riesz means on some Sobolev type spaces ]{Almost everywhere
convergence of Bochner-Riesz means on some Sobolev type spaces}
\author{ Dashan Fan and Fayou Zhao$^*$ }
\address[Dashan Fan]{Department of Mathematical Sciences, University of
Wisconsin-Milwaukee, Milwaukee, WI 53201, USA}
\email{fan@uwm.edu}
\address[Fayou Zhao]{Department of Mathematics,
Shanghai University, Shanghai 200444, P.R. China}
\email{fyzhao@shu.edu.cn}
\thanks{ The research was supported by National Natural Science Foundation
of China (Grant Nos. 11471288, 11201287) and China Scholarship Council
(Grant No. 201406895019).\\
* Corresponding author.}

\begin{abstract}
In this paper, we investigate the convergence of the Bochner-Riesz means on
some Sobolev type spaces including $L^p$-Sobolev spaces $(p\geq 1)$ and $H^q$-Sobolev
spaces $(0<q<1)$. The relation between the smoothness imposed on functions and the rate of
almost everywhere convergence of the generalized Bochner-Riesz means is
given.
\end{abstract}

\maketitle

\section{Introduction}

Let $R>0.$ We consider the generalized Bochner-Riesz means ${S_{R}^{\delta
,\gamma }}$ on the Euclidean space $\mathbb{R}^{n}$ defined via the Fourier
transform by
\begin{equation*}
{(S_{R}^{\delta ,\gamma }f)}^{\wedge }(\xi )=\left( 1-\frac{|\xi |^{\gamma }%
}{R^{\gamma }}\right) _{+}^{\delta }\widehat{f}(\xi ),
\end{equation*}%
where $\delta $ and $\gamma $ are two real numbers satisfying $\delta >-1$
and $\gamma >0$.  Also, we may initially assume that $f$ are functions in
the Schwartz class $\mathscr{S}(\mathbb{R}^{n})$. Especially, $S_{R}^{\delta
,2}$ is the classical Bochner-Riesz means which was studied by many authors
[1--7, 12, 13, 17, 20, 27]. 
 The associated maximal operator of \ ${S_{R}^{\delta ,\gamma }}$ is
defined by
\begin{equation*}
(S_{\ast }^{\delta ,\gamma }f)(x)=\sup_{R>0}|(S_{R}^{\delta ,\gamma }f)(x)|.
\end{equation*}

The study of convergence of $S_{R}^{\delta ,2}$ is a long time standing
subject in the classical theory of Fourier analysis. The number $\delta
=(n-1)/2$ is called the critical index, since $\delta >{(n-1)}/{2},$
$\lim_{R\rightarrow \infty }S_{R}^{\delta ,2}f(x)=f(x)$, a.e. for any $f\in
L^{1}(\mathbb{R}^{n}),$ while Stein \cite{S3} found an $L^{1}(\mathbb{R}^{n})
$ function $f$ for which $\lim \sup_{R\rightarrow \infty
}|S_{R}^{(n-1)/2,2}f(x)|=\infty $ a.e. Finding a suitable subspace of $L^{1}(%
\mathbb{R}^{n})$ related to a.e. convergence of $S_{R}^{(n-1)/2,2}(f)(x)$
thus is an interesting problem. Please see the related work of Stein \cite%
{S4}, R. Fefferman \cite{F}, Lu, Taibleson and Weiss \cite{LTW}, Lu and Wang
\cite{LW}.

In order to describe our motivation, we mention two theorems related to this paper. Stein, Taibleson and G. Weiss \cite{STW} considered the
boundedness of $S_{R}^{\delta ,2}$ on the Hardy space $H^{p}(\mathbb{R}^{n})$
as follows:\newline
\noindent \textbf{Theorem A.} \ Let $0<p<1$ and $\delta _{p}=n/p-(n+1)/2$.
If $f\in H^{p}(\mathbb{R}^{n})$, then
\begin{equation*}
\left\vert \left\{ x\in \mathbb{R}^{n}:S_{\ast }^{\delta
_{p},2}(f)(x)>s\right\} \right\vert \preceq \left( \frac{\Vert f\Vert
_{H^{p}(\mathbb{R}^{n})}}{s}\right) ^{p},
\end{equation*}%
for all $s>0.$ Thus, $\lim_{R\rightarrow \infty }S_{R}^{\delta
_{p},2}f(x)=f(x)$, a.e. for any $f\in H^{p}(\mathbb{R}^{n})$.

Let  $m$ be a positive integer and let $ L_m^{1}(\mathbb{R}^{n})$  denote the inhomogeneous Sobolev space of
all functions  $f$ satisfying  $\partial^{\alpha}f\in  L^{1}(\mathbb{R}^{n})$ for all multi-indicies
$\alpha$ with $|\alpha|\leq m$. Wang \cite{Ws} showed
the convergence and approximation for functions in spaces $L_m^{1}(\mathbb{R}^{n})$ by
Bochner-Riesz means $S_{R}^{\delta ,2}$ below the critical index.

\noindent \textbf{Theorem B.} \ Suppose $f\in  L_m^{1}(\mathbb{R}^{n})$,
$m\in \mathbb{N}\backslash \{0\}$ and ${(n-1)}/{2}-m\geq
-1 $. Then
\begin{equation*}
(S_{R}^{\delta ,2}f)(x)-f(x)=\left\{
\begin{array}{lll}
o\left( R^{\frac{n-1}{2}-m-\delta }\right) , & \frac{n-1}{2}-m<\delta <\frac{%
n-1}{2}-m+1, & m\geq 1, \\
o\left( R^{-1}\log R\right) , & \delta =\frac{n-1}{2}-m+1, & m\geq 1, \\
o\left( R^{\frac{n-1}{2}-m-\delta }\right) , & \frac{n-1}{2}-m+1<\delta <%
\frac{n-1}{2}-m+2, & m\geq 2, \\
O\left( R^{-2}\log R\right) , & \delta =\frac{n-1}{2}-m+2, & m\geq 2, \\
O\left( R^{-2}\right) , & \delta >\frac{n-1}{2}-m+2, & m\geq 2,%
\end{array}%
\right.
\end{equation*} in the sense of almost everywhere convergence.

 In our recent work \cite{FZ}, we study the convergence of the generalized Bochner-Riesz means
 $S_{R}^{\delta,\gamma }$ with $\delta=(n-1)/2$ on the
block-Sobolev spaces. The relation between the smoothness imposed on blocks and the rate of almost everywhere convergence of the generalized
Bochner-Riesz means at the critical index is given.
This furnishes us a motivation for studying the generalized Bochner-Riesz
means on some corresponding Sobolev type spaces. We want mention that the study of generalized Bochner-Riesz means $S_{R}^{\delta,\gamma }$ is
not merely a simple extension of using $\gamma$  to replace 2, and it is naturally raised from the approximation theory in order to enhance the saturation of the operator (see Theorem \ref{t3}). Our aim of this paper is not only to analyze the convergence of the Bochner-Riesz means $%
S_{R}^{\delta ,\gamma }f$ on certain Sobolev type spaces including $L^{p}$%
-Sobolev spaces and $H^{p}$-Sobolev spaces, but also to obtain the relation
between the smoothness imposed on functions and the rate of almost
everywhere convergence of $S_{R}^{\delta ,\gamma }f$.

Let $I_{\lambda }$ denote the Riesz potential operators of
order $\lambda$ on $\mathbb{R}^{n}$
for $\lambda \in \mathbb{R}$, which may act on functions or tempered
distributions. The Fourier transform of a Schwartz function (or even a
tempered distribution $f$) satisfies $(I_{\lambda }f)^{\wedge }(\xi )=|\xi
|^{-\lambda }\widehat{f}(\xi )$. If $X$ is any function space or a space of
tempered distributions, one can define Sobolev spaces based on $X$ using $%
I_{\lambda }$, $I_{\lambda }(X)$, to be the image of $X$ under $I_{\lambda }$%
. By this definition, $I_{\lambda }(L^{p})(\mathbb{R}^{n})$ are the
classical homogeneous Sobolev spaces for $p\geq 1$ (see \cite[p.16]{G1}), and $I_{\lambda }(H^{p})(\mathbb{R}%
^{n})$ are the Hardy-Sobolev spaces for $0<p\leq 1$ (see \cite{Str}). We recall that this
notation $I_{\lambda }(X)$ was used by Strichartz in \cite{Str1,Str}. By the
definition (following the terminology in \cite{Str}), $f\in I_{\lambda
}(H^{p})(\mathbb{R}^{n})$ if and only if $I_{-\lambda }f\in H^{p}(\mathbb{R}%
^{n})$.  It follows from Chapter 3 in \cite[p.338]{S2} that $g\in H^{p}(\mathbb{R}^{n})$  has an atomic
decomposition
\begin{equation*}
g=\sum_{k}c_{k}a_{k},\quad
\Vert g\Vert _{H^{p}(\mathbb{R}^{n})}^{p}\approx \sum_{k}|c_{k}|^{p}
<\infty,
\end{equation*}%
where each $a_{k}$ is a $(p,2)$-atom. Here, we call a function $a$ a $(p,2)$%
-atom, $0<p\leq 1$, if $a$ satisfies:

$(i)$\ $a$ is supported on a cube $Q$; (support condition)

$(ii)$\  $\ \Vert a\Vert _{L^{2}(\mathbb{R}^{n})}\leq |Q|^{1/2-1/p};$ (size condition)

$(iii)$\ $\ \int_{\mathbb{R}^{n}}a(x)x^{\alpha }dx=0,$ $|\alpha |\leq \lbrack
n(1/p-1)].$ (cancellation condition)

\medskip

We now formulate our main results. They are new even when $\gamma=2$.

\begin{thm}
\label{t3} \ Let $0\leq \lambda \leq \gamma $, $0<p<1$ and $\delta
_{p}=n/p-(n+1)/2$. If $f\in I_{\lambda }(H^{p})(\mathbb{R}^{n})$, then for $%
0\leq \lambda <\gamma $,
\begin{equation*}
(S_{R}^{\delta _{p},\gamma }f)(x)-f(x)=o(1/{R^{\lambda }})\ a.e.\ as\
R\rightarrow \infty ;
\end{equation*}%
and for $\lambda =\gamma $,
\begin{equation*}
(S_{R}^{\delta _{p},\gamma }f)(x)-f(x)=O(1/{R^{\gamma }})\ a.e.\ as\
R\rightarrow \infty .
\end{equation*}Moreover, this rate is sharp in the sense that
$O(1/{R^{\gamma }})$ can not be replaced by $o(1/{R^{\gamma }})$.
\end{thm}

Let ${M}_{\lambda }^{\delta_p ,\gamma }f$ be the maximal function defined by
\begin{equation*}
({M}_{\lambda }^{\delta_p ,\gamma }f)(x)=\sup_{R>0}|R^{\lambda
}\{(S_{R}^{\delta_p ,\gamma }f)(x)-f(x)\}|.
\end{equation*}%
To prove Theorem \ref{t3}, we will need to make use of the following
estimate on ${M}_{\lambda }^{\delta_p ,\gamma }f$.

\begin{thm}
\label{t2} \ Let $0\leq \lambda\leq \gamma$, $0<p<1$ and $%
\delta_p=n/p-(n+1)/2$. If $f\in I_{\lambda}(H^p)(\mathbb{R}^n)$, then
\begin{equation*}
\left|\left\{x\in\mathbb{R}^n: ({M}_{\lambda }^{\delta_p ,\gamma
}f)(x)>s\right\}\right|\preceq \left(\frac{\|f\|_{I_{\lambda}(H^p)(\mathbb{R}%
^n)}}{s}\right)^p,
\end{equation*}
for all $s>0$.
\end{thm}
It is worth pointing out that if $\lambda=0$, one may follow the method of \cite{STW} to yield that
for all $s>0$ \begin{equation*}
\left\vert \left\{ x\in \mathbb{R}^{n}:S_{\ast }^{\delta
_{p},\gamma}(f)(x)>s\right\} \right\vert \preceq \left( \frac{\Vert f\Vert
_{H^{p}(\mathbb{R}^{n})}}{s}\right) ^{p},\  \gamma\geq 0
\end{equation*} whenever $f\in H^{p}(\mathbb{R}^{n})$. Thus Theorem \ref{t2}
holds for the case $\lambda=0$ on account of  Chebyshev's inequality and the following fact
for $0<p<\infty$\begin{equation*}
\|f\|_{L^p(\mathbb{R}^n)}=\|\lim_{t\rightarrow 0^{+}}(e^{-t|x|^2}\ast f)\|_{L^p(\mathbb{R}^n)}
\leq \|f\|_{H^p(\mathbb{R}^n)}.
\end{equation*}
%
%

\medskip For the case  $0<\delta <\frac{n-1}{2},$ we use a different approach
from \cite{Ws} and obtain the following
result.

\begin{thm}\label{ts}
\label{t1} Let $0<\delta <(n-1)/2$, $1\leq p<\frac{2n}{n-1-2\delta} $  and ${(n-1)%
}/2-\delta <\lambda \leq \gamma $.  If $f\in I_{\lambda }(L^{p})(%
\mathbb{R}^{n})$, then we have
\begin{equation*}
(S_{R}^{\delta ,\gamma }f)(x)-f(x)=o(1/{R^{\lambda +\delta -\frac{n-1}{2}}}%
)\ a.e.\ as\ R\rightarrow \infty .
\end{equation*}
\end{thm}

Here we make some further comments about Theorem \ref{ts}.
If $p=1$, then the inhomogeneous space $L_{m}^1(\mathbb{R}^n)$ is a proper subspace of the space $I_{m}(L^1)(\mathbb{R}^n)$ when $m$ is even. So we may compare
Theorem B with Theorem \ref{ts} ($\gamma=2$) when $m$ is even and
$0<\frac{n-1}{2}-\delta\leq 2$. If $m+\delta-\frac{n-1}{2}=1$ and $f\in L_m^1(\mathbb{R}^n)$, then
Theorem B gives the rate $o(R^{-1}\log R)$ while the rate in Theorem \ref{ts} is $o(R^{-1})$.
Clearly, this theorem is a substantial extension of Theorem B when $0<\delta<{(n-1)}/2$.

Our idea is mainly inspired from our previous work \cite{FZ}. Precisely,
observe that the Fourier transform of $R^{\lambda }\{(S_{R}^{\delta ,\gamma
}f)-f\}$ is $\mu (\cdot /R)$ $\widehat{g},$ where $g=I_{-\lambda }f$ and
the multiplier $\mu $ is given by
\begin{equation*}
\mu (\xi )=\frac{(1-\left\vert \xi \right\vert ^{\gamma })_{+}^{\delta }-1}{%
\left\vert \xi \right\vert ^{\lambda }},\ \xi \neq 0\text{ \ and }\mu
(0)=\lim_{t\rightarrow 0^{+}}\frac{(1-t^{\gamma })_{+}^{\delta }-1}{%
t^{\lambda }}.
\end{equation*}%
Hence we will decompose the multiplier $\mu $ as a sum of \ $\mu _{0},$ $\mu
_{1},$ $\mu _{\infty },$ centralizing at $0,$ $1$ and near $\infty ,$
respectively. Then each corresponding kernel will be carefully estimated.

The rest of the paper is organized as follows. In Section 2, we define some necessary
notations and definitions that will be used throughout this paper. In this section, we also
prove three lemmas which will be used in the next section. In Section 3, we prove
 the almost everywhere convergence of the Bochner-Riesz means $%
S_{R}^{\delta_p ,\gamma }f$ on $H^{p}$-Sobolev spaces, where $0<p<1$. In the last section,
we present the almost everywhere convergence of the Bochner-Riesz means
$S_{R}^{\delta ,\gamma }f$ on  $L^{p}$-Sobolev spaces below the critical index for
$1\leq p<\frac{2n}{n-1-2\delta}$.

Throughout this article, we will use the symbol $A\preceq B$ to mean that
there exists a constant $C>0$ independent of all essential variables such
that $A\leq CB$. We also use the notation $A\simeq B$ if $A\preceq B$ and $%
B\preceq A$. The symbol $A\approx B$ means that there exists a constant $C$ independent of all essential variables such
that $A= CB$.

\section{Some Preliminaries}

\label{s2} We introduce the requisite notation. For $x=(x_1,x_2,\ldots,
x_n)\in \mathbb{R}^n$, the first partial derivative of a function $f$ on $%
\mathbb{R}^n$ with respect to the $j$th variable $x_j$ is denoted by $%
\partial_j f$ while the $m$th partial derivative with respect to the $j$th
variable is denoted by $\partial_j^m f$. A multi-index $\alpha$ is an
ordered $n$-tuple of nonnegative integers. For a multi-index $%
\alpha=(\alpha_1,\ldots, \alpha_n)$, $\partial_{x}^{\alpha} f$ or $\partial^{\alpha} f$
 denotes the derivative $\partial_1^{\alpha_1}\cdots \partial_n^{\alpha_n} f$. Let $%
|\alpha|=\alpha_1+\ldots+\alpha_n$ and $\alpha!=\alpha_1 ! \cdots \alpha_n!$%
. The number $|\alpha|$ indicates the total order of differentiation of $%
\partial_{x}^{\alpha}f$. For $x\in \mathbb{R}^n$ and $\alpha=(\alpha_1,%
\ldots, \alpha_n)$ a multi-index, we set $x^{\alpha}=x_1^{\alpha_1}\cdots
x_n^{\alpha_n}$. Let $\beta=(\beta_1,\ldots,\beta_n)$ be a multi-index. The
notation $\alpha\leq \beta$ means that $\beta$ ranges over all multi-indices
satisfying $0\leq \alpha_j\leq \beta_j $ for all $1\leq j\leq n$. For more
details, see \cite[p.108]{G1}. Let $[r]$ be the greatest integer less than or equal to the real number $r$.

Let $\phi_{0}, \phi_{1}, \phi_{\infty}\in C^{\infty}(\mathbb{R}^n)$ be
radial functions satisfying the following conditions: $(i) \ \phi_0$ is
supported on the set $\{\xi \in \mathbb{R}^n: |\xi|<1/2\}$, $\phi_1$ is
supported on the annulus $\{\xi \in \mathbb{R}^n: 1/4<|\xi|< 2\}$, and $%
\phi_\infty$ is supported on the set $\{\xi\in \mathbb{R}^n: |\xi|>3/2\}$; $%
(ii)\ \phi_{0}(\xi)+\phi_{1}(\xi)+\phi_{\infty}(\xi)=1,$ $0\leq
\phi_{j}(\xi)\leq 1$ for $j=0,1, \infty$, and
\begin{align*}
\phi_0(\xi)=\left\{
\begin{array}{ll}
1 & |\xi|\leq {1}/{4}, \\
0 & |\xi|\geq {1}/{2};%
\end{array}%
\right. \quad \phi_1(\xi)=\left\{
\begin{array}{ll}
1 & {1}/{2}\leq |\xi|\leq {3}/{2}, \\
0 & |\xi|\leq {1}/{4}\ \text{and}\ |\xi|\geq 2;%
\end{array}%
\right.
\end{align*}%
\begin{align*}
\phi_\infty(\xi)=\left\{
\begin{array}{ll}
1 & |\xi|\geq2, \\
0 & |\xi|\leq {3}/{2}.%
\end{array}%
\right.
\end{align*}
Let $\Psi_{\infty}(\xi)=\phi_1(\xi)+\phi_{\infty}(\xi)$. Obviously, $%
\Psi_{\infty}\in C^{\infty}(\mathbb{R}^n)$ is supported on the set $\{\xi:
|\xi|>1/4\}$ satisfying $0\leq \Psi_{\infty}(\xi)\leq 1$ and $%
\Psi_{\infty}(\xi)=1$ on the set $\{\xi: |\xi|\geq 1/2\}$.

Assume that the kernel $B_{R}^{\delta ,\gamma }$ of $S_{R}^{\delta ,\gamma }$ is defined
by
\begin{equation*}
B_{R}^{\delta ,\gamma }(x)=\int_{\mathbb{R}^{n}}\left( 1-\frac{|\xi
|^{\gamma }}{R^{\gamma }}\right) _{+}^{\delta }e^{2\pi i\xi \cdot x}d\xi .
\end{equation*}%
We rewrite $S_{R}^{\delta ,\gamma }$ as a convolution operator
\begin{equation*}
S_{R}^{\delta ,\gamma }f=B_{R}^{\delta ,\gamma }\ast f.
\end{equation*}%
Let $g=:I_{-\lambda }f$. Then the Fourier transform of the function $%
R^{\lambda }\{(S_{R}^{\delta ,\gamma }f)-f\}$ can be written as $%
m_{\lambda ,0}^{\delta,\gamma }(\cdot /R)\widehat{g}+m_{\lambda
,1}^{\delta ,\gamma }(\cdot /R)\widehat{g}-m_{\lambda ,\infty }^{\delta
,\gamma }(\cdot /R)\widehat{g}$, where
\begin{align*}
& m_{\lambda ,0}^{\delta ,\gamma }(\xi )=\frac{\phi _{0}\left( {\xi }%
\right) }{|\xi |^{\lambda }}\left( \left( 1-{|\xi |^{\gamma }}\right)
_{+}^{\delta}-1\right) ,m_{\lambda ,1}^{\delta ,\gamma }(\xi )=%
\frac{\phi _{1}\left( {\xi }\right) }{|\xi |^{\lambda }}\left( 1-|\xi
|^{\gamma }\right) _{+}^{\delta},\text{ \ \ \ \ \ } \\
& m_{\lambda ,\infty }^{\delta,\gamma }(\xi )=\frac{\Psi _{\infty }(\xi
)}{|\xi |^{\lambda }}.
\end{align*} Here and in what follows, for simplicity, we define the
operators
\begin{equation*}
(T_{R,\lambda ,j}^{\delta,\gamma }g)^{\wedge }(\xi )=m_{\lambda
,j}^{\delta ,\gamma }(\xi /R)\widehat{g}(\xi ),
\end{equation*}%
and denote by $(K_{\lambda ,j}^{\delta ,\gamma
})_{1/R}(x)=R^{n}K_{\lambda ,j}^{\delta ,\gamma }(Rx)$ the kernel of $%
T_{R,\lambda ,j}^{\delta,\gamma }$, where
\begin{equation*}
K_{\lambda ,j}^{\delta ,\gamma }(x)={m_{\lambda ,j}^{\delta,\gamma }%
}^{\vee }(x),\ j=0,1,\infty .
\end{equation*}%
Then each operator $T_{R,\lambda ,j}^{\delta ,\gamma }g$ can be written
as a convolution operator:
\begin{equation*}
T_{R,\lambda ,j}^{\delta ,\gamma }g(x)=(K_{\lambda ,j}^{\delta,\gamma })_{1/R}\ast g(x),
\end{equation*}%
for $j=0,1,\infty $.

In the following three lemmas of this section, we will estimate
$\partial_{x}^{\beta}K_{\lambda,j}^{\delta_p,\gamma}(x)$, $j=0,1,\infty $,
where
\begin{equation*}
K_{\lambda ,j}^{\delta _{p},\gamma }(x)={m_{\lambda ,j}^{\delta _{p},\gamma }%
}^{\vee }(x)=\int_{%
\mathbb{R}
^{n}}{m_{\lambda ,j}^{\delta _{p},\gamma }}(\xi )e^{2\pi i\xi \cdot x}d\xi,
\end{equation*} here and in what follows, the notation $\delta_p$ is $\delta
_{p}=n/p-(n+1)/2$ for $0<p<1$. To get the desired estimates for the case $j=0, \infty$,
 we consider cones $E_{k},\ k=1,2,\ldots ,n$,
defined by
\begin{equation*}
E_{k}=\{x\in \mathbb{R}^{n}:|x_{k}|^{2}\geq |x|^{2}/{2n}\}
\end{equation*}%
and the smaller cones
\begin{equation*}
E_{k}^{0}=\{x\in \mathbb{R}^{n}:|x_{k}|^{2}\geq |x|^{2}/{n}\}.
\end{equation*}%
Then
\begin{equation*}
\bigcup_{k=1}^{n}E_{k}^{0}=\mathbb{R}^{n}.
\end{equation*}%
Let $\chi _{E_{k}}$ be\ a nonnegative $C^{\infty }(\mathbb{R}^{n})$ function
supported on $E_{k}$ and let $\chi _{E_{k}}(x)=1$ for $x\in E_{k}^{0}$ and $0\leq \chi _{E_{k}}(x)\leq 1$. Clearly,
we have
\begin{equation*}
|\partial_{x}^\beta K_{\lambda ,j}^{\delta _{p},\gamma }(x)|\leq \sum_{k=1}^{n}\left\vert \chi
_{E_{k}}(x)\partial_{x}^\beta \left({m_{\lambda ,j}^{\delta _{p},\gamma }}^{\vee }(x)\right)\right\vert ,\
j=0,\infty .
\end{equation*}

\begin{lem}
\label{lem2} For $0\leq \lambda\leq \gamma$ and $0<p<1$, we have
\begin{equation*}
|\partial_{x}^{\beta}K_{\lambda,0}^{\delta_p,\gamma}(x)|\preceq \frac{1}{%
(1+|x|)^{n+|\beta|+\gamma-\lambda}}.
\end{equation*}
\end{lem}

\begin{proof}
A direct calculation shows that
\begin{equation*}
\partial_{x}^{\beta}K_{\lambda,0}^{\delta_p,\gamma}(x)=\int_{\mathbb{R}^n}%
\frac{\phi_0 \left( {\xi}\right) }{|\xi|^{\lambda }}\left( \left( 1-{%
|\xi|^{\gamma }}\right) _{+}^{\delta_p }-1\right)(2\pi i\xi)^{\beta} e^{2\pi
i\xi\cdot x}d\xi.
\end{equation*}

By Taylor's expansion we notice that in the support of $\phi_0 $,
\begin{equation*}
(1-|\xi |^{\gamma })_{+}^{\delta_p }-1=\delta_p |\xi |^{\gamma }+\omega (\xi
),
\end{equation*}%
where $\omega $ is a $C^{\infty }$ function on $\mathbb{R}^{n}\backslash
\{0\}$ satisfying $\omega (\xi )=O(|\xi |^{2\gamma })$ and $\left( \nabla
^{k}\omega \right) (\xi )=O(|\xi |^{2\gamma -k})$ for $k\in \mathbb{N}$ as $|\xi |\rightarrow 0$.
Thus we only need to work the term
\begin{equation*}
\mathfrak{K}_0(x)=\int_{\mathbb{R}^n}\phi_0 (\xi) |\xi|^{\gamma-\lambda}
\xi^{\beta} e^{2\pi i\xi\cdot x}d\xi,
\end{equation*}
since the other terms are similarly estimated.

To estimate the function $\mathfrak{K}_{0}(x)$, we break up $E_{1}$ into the
regions $\{x\in E_{1}:|x|<1\}$ and $\{x\in E_{1}:|x|\geq 1\}$. \ Clearly, we
have
\begin{equation*}
|\mathfrak{K}_{0}(x)|\preceq 1,\ \text{if}\ |x|<1.
\end{equation*}

We now estimate the second case $|x|>1$. Recalling that
\begin{equation*}
|\mathfrak{K}_{0}(x)|\leq \sum_{k=1}^{n}\left\vert \chi _{E_{k}}(x)\int_{%
\mathbb{R}^{n}}\phi _{0}(\xi )|\xi |^{\gamma -\lambda }\xi ^{\beta }e^{2\pi
i\xi \cdot x}d\xi \right\vert ,
\end{equation*}%
by the symmetry, we only need to work
\begin{equation*}
\chi _{E_{1}}(x)\int_{\mathbb{R}^{n}}\phi _{0}(\xi )|\xi |^{\gamma -\lambda
}\xi ^{\beta }e^{2\pi i\xi \cdot x}d\xi .
\end{equation*}


Denote by $\vec{e}_{1}=(1,0,\ldots ,0)$. We integrate
for $n$ times on the $\xi _{1}$ variable to obtain
\begin{align*}
& \quad \quad \chi _{E_{1}}(x)\int_{\mathbb{R}^{n}}\phi _{0}(\xi )|\xi
|^{\gamma -\lambda }\xi ^{\beta }e^{2\pi i\xi \cdot x}d\xi  \\
& \quad =O\left( \frac{\chi _{E_{1}}(x)}{|x|^{n}}\int_{\mathbb{R}^{n}}%
\widetilde{\phi _{0}}(\xi )\left( \sum_{k=0}^{n}c_{k}|\xi |^{\gamma -\lambda
-2k}\xi ^{\beta +(2k-n)\vec{e}_{1}}\right) e^{2\pi i\xi \cdot x}d\xi \right)
,
\end{align*}%
as $|x|>1,$ where, without loss of generality, we may assume that $%
\widetilde{\phi _{0}}$ are different functions in each occurrence satisfying
the following three conditions: (i) $\widetilde{\phi _{0}}\in C^{\infty }(%
\mathbb{R}^{n}\backslash \{0\}),$\ (ii) $\mathrm{supp}\ \widetilde{\phi _{0}}%
\subseteq \{\xi \in \mathbb{R}^{n}:|\xi |\leq 1/2\}$, (iii) $\widetilde{\phi
_{0}}(\xi)=O(|\xi |^{n})$, as $|\xi |\rightarrow 0$. Now let $\eta \in C^{\infty
}(\mathbb{R}^{n})$ be a radial function satisfying that $\eta (x)=1$ for $%
|x|\leq 1$ and $\eta (x)=0$ for $|x|\geq 2$. For $0\leq k\leq n$, write
\begin{align*}
& \quad \quad \quad \chi _{E_{1}}(x)\int_{\mathbb{R}^{n}}\widetilde{\phi _{0}%
}(\xi )|\xi |^{\gamma -\lambda -2k}\xi ^{\beta +(2k-n)\vec{e}_{1}}e^{2\pi
i\xi \cdot x}d\xi  \\
& \quad \quad \quad \quad =\chi _{E_{1}}(x)\int_{\mathbb{R}^{n}}\eta (|x|\xi
)\widetilde{\phi _{0}}(\xi )|\xi |^{\gamma -\lambda -2k}\xi ^{\beta +(2k-n)%
\vec{e}_{1}}e^{2\pi i\xi \cdot x}d\xi  \\
& \quad \quad \quad \quad \quad +\chi _{E_{1}}(x)\int_{\mathbb{R}^{n}}\left(
1-\eta (|x|\xi )\right) \widetilde{\phi _{0}}(\xi )|\xi |^{\gamma -\lambda
-2k}\xi ^{\beta +(2k-n)\vec{e}_{1}}e^{2\pi i\xi \cdot x}d\xi .
\end{align*}%
Then we have
\begin{align*}
& \quad \quad \quad \left\vert \frac{\chi _{E_{1}}(x)}{|x|^{n}}\int_{\mathbb{%
R}^{n}}\widetilde{\phi _{0}}(\xi )|\xi |^{\gamma -\lambda -2k}\xi ^{\beta
+(2k-n)\vec{e}_{1}}e^{2\pi i\xi \cdot x}d\xi \right\vert  \\
& \quad \quad \quad \quad \preceq \frac{1}{|x|^{n}}\int_{|\xi |<2/{|x|}}|\xi
|^{\gamma -\lambda +|\beta |-n}d\xi  \\
& \quad \quad \quad \quad \quad +\frac{\chi _{E_{1}}(x)}{|x|^{n}}\left\vert
\int_{\mathbb{R}^{n}}\left( 1-\eta (|x|\xi )\right) \widetilde{\phi _{0}}%
(\xi )|\xi |^{\gamma -\lambda -2k}\xi ^{\beta +(2k-n)\vec{e}_{1}}e^{2\pi
i\xi \cdot x}d\xi \right\vert .
\end{align*}%
The first term
\begin{equation*}
\frac{1}{|x|^{n}}\int_{|\xi |<2/{|x|}}|\xi |^{\gamma -\lambda +|\beta
|-n}d\xi \preceq |x|^{\lambda -\gamma -\left\vert \beta \right\vert -n}.
\end{equation*}%
For the second term, we use the integration by parts one more time on the $%
\xi _{1}$ variable to obtain,
{\allowdisplaybreaks\begin{align*}
& \quad \frac{\chi _{E_{1}}(x)}{|x|^{n}}\left\vert \int_{\mathbb{R}%
^{n}}\left( 1-\eta (|x|\xi )\right) \widetilde{\phi _{0}}(\xi )|\xi
|^{\gamma -\lambda -2k}\xi ^{\beta +(2k-n)\vec{e}_{1}}e^{2\pi i\xi \cdot
x}d\xi \right\vert  \\
& \preceq \frac{1}{|x|^{n+1}}\left\vert \int_{\mathbb{R}^{n}}\left( 1-\eta
(|x|\xi )\right) \widetilde{\phi _{0}}(\xi )|\xi |^{\gamma -\lambda
-2k-2}\xi ^{\beta +(2k-n+1)\vec{e}_{1}}e^{2\pi i\xi \cdot x}d\xi \right\vert
\\
& \quad +\frac{1}{|x|^{n+1}}\left\vert \int_{\mathbb{R}^{n}}\left( 1-\eta
(|x|\xi )\right) \widetilde{\phi _{0}}(\xi )|\xi |^{\gamma -\lambda -2k}\xi
^{\beta +(2k-n-1)\vec{e}_{1}}e^{2\pi i\xi \cdot x}d\xi \right\vert  \\
& \quad +\frac{1}{|x|^{n}} \int_{1/{|x|<|\xi |<2/{|x|}}}|\xi
|^{\gamma -\lambda +|\beta |-n}d\xi   \\
& \preceq \frac{1}{|x|^{n+|\beta |+\gamma -\lambda }},
\end{align*}}where the function $\partial_1^1 \eta$ is supported in the annulus
$\{\xi \in \mathbb{R}^{n}:1<|\xi |<2\}$.
\end{proof}

\begin{lem}
\label{lem21} For $\lambda\geq 0$ and $0<p<1$, we have
\begin{equation*}
|\partial_{x}^{\beta}K_{\lambda,1}^{\delta_p,\gamma}(x)|\preceq 1, \text{\
if\ } |x|\leq 1,
\end{equation*}
and
\begin{equation*}
|\partial_{x}^{\beta}K_{\lambda,1}^{\delta_p,\gamma}(x)|\preceq |x|^{-\frac{%
n+1}{2}-\delta_{p}}, \text{\ if\ }|x|> 1.
\end{equation*}
\end{lem}

\begin{proof}
Using the formula in \cite[p.155]{SW}, we know that
\begin{equation*}
K_{\lambda ,1}^{\delta _{p},\gamma }(x)=(2\pi )^{\frac{n}{2}}\int_{1/{4}}^{1}%
{\phi _{1}\left( {t}\right) }\left( 1-t^{\gamma }\right) ^{\delta
_{p}}t^{n-\lambda -1}V_{\frac{n-2}{2}}(2\pi |x|t)dt,  \label{k111}
\end{equation*}%
where
\begin{equation*}
V_{\frac{n-2}{2}}(2\pi |x|t)=\frac{J_{\frac{n-2}{2}}(2\pi |x|t)}{(2\pi
|x|t)^{\frac{n-2}{2}}},
\end{equation*}   $J_{v}(t)$ denotes the usual Bessel function satisfying that
$v$ is a complex number  $\text{Re}\ v>-1/2$ and $t\geq 0$.
Using the formula
\begin{equation*}
\frac{d}{dt}\left( \frac{J_{\gamma }(t)}{t^{\gamma }}\right) =-\frac{%
J_{\gamma +1}(t)}{t^{\gamma }},
\end{equation*}%
for any multi-index $\beta $ and $\lambda \geq 0$, an easy computation shows
that
\begin{align}
& \partial _{x}^{\beta }V_{\frac{n-2}{2}}(2\pi |x|t)=(-1)^{|\beta |}(2\pi
t)^{2|\beta |}x^{\beta }V_{\frac{n-2}{2}+|\beta |}(2\pi |x|t)  \label{pk1} \\
& \quad \quad \quad \quad \quad \quad \quad \quad \quad +\sum_{0<|k|\leq
\lbrack \frac{|\beta |}{2}]}c_{k}t^{2(|\beta |-|k|)}x^{\beta -2k}V_{\frac{n-2%
}{2}+|\beta |-|k|}(2\pi |x|t). \notag
\end{align}  Invoking the above estimate \eqref{pk1}, we only
need to estimate the term
\begin{equation}
\mathfrak{K}_{1}(x)=x^{\beta }\int_{1/{4}}^{1}{\phi _{1}\left( {t}\right) }%
\left( 1-t^{\gamma }\right) ^{\delta _{p}}t^{n+2|\beta |-\lambda -1}V_{\frac{%
n-2}{2}+|\beta |}(2\pi |x|t)dt.  \label{k11}
\end{equation}%
As a result, the estimate $\partial _{x}^{\beta }K_{\lambda ,1}^{\delta
_{p},\gamma }(x)$ can be obtained because the other terms can be handled in
a similar method.

To estimate \eqref{k11}, we need to use the behavior of the Bessel function $%
J_{v}(t) $. The following estimate and asymptotic formula can be found in \cite%
[p.430]{G1} and \cite[p.338]{S2}
\begin{equation}
|J_{v}(t)|\preceq t^{\text{Re}\ v},\ 0<t\leq 1;  \label{bs1}
\end{equation}
when $v$ is fixed, the asymptotic formula for $J_{v}(t)$, as $t\rightarrow
\infty $, is%
\begin{eqnarray*}\label{bs2}
\notag J_{v}(t) &\sim &t^{-1/2}e^{it}\sum_{j=0}^{\infty
}c_{j}t^{-j}+t^{-1/2}e^{-it}\sum_{j=0}^{\infty }d_{j}t^{-j} \\
&=&t^{-1/2}e^{it}\sum_{j=0}^{N_{0}}c_{j}t^{-j}+t^{-1/2}e^{-it}%
\sum_{j=0}^{N_{0}}d_{j}t^{-j}+O\left( \frac{1}{t^{N_{0}+3/2}}\right)
\end{eqnarray*}
where $c_{j}$, $d_{j}$ are coefficients, and the positive integer $N_{0}$
will be determined later.

We now divide the value of $\ |x|$ into two cases: $|x|\leq 1$ and $|x|>1$.
Using the estimate of \eqref{bs1}, we have that for $|x|\leq 1$,
\begin{equation*}
|V_{\frac{n-2}{2}}(2\pi |x|t)|\preceq 1,
\end{equation*}%
and
\begin{align*}
\quad \left\vert \mathfrak{K}_{1}(x)\right\vert & =\left\vert x^{\beta
}\int_{1/{4}}^{1}{\phi _{1}\left( {t}\right) }\left( 1-t^{\gamma }\right)
^{\delta _{p}}t^{n+2|\beta |-\lambda -1}V_{\frac{n-2}{2}+|\beta |}(2\pi
|x|t)dt\right\vert  \\
& \preceq {|x|}^{|\beta |}\int_{1/4}^{1}(1-t^{\gamma })^{\delta
_{p}}dt\preceq 1.
\end{align*}

For $|x|>1$, applying the asymptotic formula for $J_{v}(t)$, it is enough to
estimate the function $G_{j}\ (0\leq j\leq N_{0})$ 
\begin{equation*}
G_{j}(x)={|x|}^{-(\frac{n-1}{2}+|\beta|+j)}x^{\beta}\int_{1/{4}}^{1}{\phi_1
\left( {t}\right) }\left( 1-t^{\gamma }\right) ^{\delta_{p}}t^{\frac{n-1}{2}%
+|\beta|-\lambda-j}e^{\pm 2\pi i|x|t}dt.
\end{equation*}

If  $\delta _{p}$ is an integer, then we take $N_{0}=\delta _{p}$. The estimate
 $|G_{N_{0}}(x)|\preceq {|x|}^{-\frac{n+1}{2}-\delta _{p}}$ is easy. For $%
0\leq j<N_{0}$, after integrating by parts $\delta _{p}-j$ times, we see
that
\begin{equation*}
G_{j}(x)\approx \frac{x^{\beta }}{|x|^{\frac{n-1}{2}+\delta _{p}+|\beta |}}%
\int_{1/{4}}^{1}\rho _{j,1}(t)\left( 1-t^{\gamma }\right) ^{j}e^{\pm 2\pi
i|x|t}dt,
\end{equation*}%
where $\rho _{j,1}(t)$ is a $C^{\infty }(\mathbb{R}^{n})$ function
satisfying $\mathrm{supp}\ \rho _{j,1}(t)\subseteq \lbrack 1/4,1]$. It is a
trivial case for  $j=0$. For $0<j<N_{0}$, using integration by parts again,
we obtain
\begin{equation*}
|G_{j}(x)|\preceq {|x|}^{-\frac{n+1}{2}-\delta _{p}},\ 0<j<N_{0}.
\end{equation*}

If $\delta_{p}$ is not an integer, then we take $N_0=[\delta_{p}]+2$. It is a simple matter to
establish $|G_{N_0}(x)|\preceq {|x|}^{-\frac{n+1}{2}-\delta_{p}}$. For $0\leq
j< N_0$, using integration by parts $[\delta_{p}]+1-j$ times, we get
\begin{equation*}
G_j(x)\approx \frac{x^{\beta}}{|x|^{\frac{n-1}{2}+[\delta_{p}]+|\beta|+1}}%
\int_{1/{4}}^1
\rho_{j,2}(t)\left(1-t^\gamma\right)^{j-([\delta_p]+1-\delta_p)} e^{\pm 2\pi
i |x|t} dt,
\end{equation*}
where $0<[\delta_p]+1-\delta_p< 1$ and $\rho_{j,2}(t)$ is a $C^{\infty}(%
\mathbb{R}^n)$ function satisfying $\mathrm{supp}\ \rho _{j,2}(t)\subseteq
[1/4,1]$. Write
{\allowdisplaybreaks\begin{align*}
&\quad\left|\frac{x^{\beta}}{|x|^{\frac{n-1}{2}+[\delta_{p}]+|\beta|+1}}%
\int_{1/{4}}^1
\rho_{j,2}(t)\left(1-t^\gamma\right)^{j-([\delta_p]+1-\delta_p)} e^{\pm 2\pi
i |x|t} dt\right| \\
&\leq \frac{1}{|x|^{\frac{n+1}{2}+[\delta_{p}]}}\left|\int_{1/4}^{1-1/{|x|}}
\rho_{j,2}(t)\left(1-t^\gamma\right)^{j-([\delta_p]+1-\delta_p)} e^{\pm 2\pi
i |x|t} dt\right| \\
&\quad+\frac{1}{|x|^{\frac{n+1}{2}+[\delta_{p}]}}\left|\int_{1-1/{|x|}}^1
\rho_{j,2}(t)\left(1-t^\gamma\right)^{j-([\delta_p]+1-\delta_p)} e^{\pm 2\pi
i |x|t} dt\right| \\
&:=G_{j,1}(x)+G_{j,2}(x).
\end{align*}}
For $G_{j,1}(x)$, using one more integration by parts, one has
\begin{equation}  \label{g11}
G_{j,1}(x)\preceq {|x|}^{-\frac{n+1}{2}-\delta_{p}}.
\end{equation}
For $G_{j,2}(x)$, we have for $|x|>1$,
\begin{align}  \label{g12}
&\quad G_{j,2}(x)\preceq\frac{1}{|x|^{\frac{n+1}{2}+[\delta_{p}]}}\int_{1-1/{%
|x|}}^1 \left(1-t^\gamma\right)^{j-([\delta_p]+1-\delta_p)} dt \\
&=-\frac{1}{\gamma(j+\delta_p-[\delta_p])}\cdot\frac{1}{|x|^{\frac{n+1}{2}%
+[\delta_{p}]}} \int_{1-1/{|x|}}^1 t^{1-\gamma}
d\left(1-t^\gamma\right)^{j+\delta_p-[\delta_p]}  \notag \\
&\preceq \frac{1}{|x|^{\frac{n+1}{2}+[\delta_{p}]}} \cdot\left(\frac{1}{|x|}%
-O\left(\frac{1}{|x|^2}\right)\right)^{j+\delta_p-[\delta_p]}  \notag \\
&\preceq {|x|}^{-\frac{n+1}{2}-\delta_{p}}.  \notag
\end{align}
The estimates of \eqref{g11} and \eqref{g12} give
\begin{equation*}
|G_j(x)|\preceq {|x|}^{-\frac{n+1}{2}-\delta_{p}}, \ 0\leq j< N_0.
\end{equation*}

Combing all the estimates, we finish the proof of Lemma \ref{lem21}.
\end{proof}

\begin{lem}
\label{lem3} For $\lambda> 0$ and $0<p<1$, we have that
\begin{equation*}
|\partial_{x}^{\beta}K_{\lambda,\infty}^{\delta_p,\gamma}(x)|\preceq \frac{1%
}{|x|^{n+|\beta|-\lambda}},\ \text{if}\ |x|<1
\end{equation*}
and if $|x|\geq 1$
\begin{equation*}
|\partial_{x}^{\beta}K_{\lambda,\infty}^{\delta_p,\gamma}(x)|\preceq \frac{1%
}{|x|^{n+L}},\ \text{for any positive integer}\ L>|\beta|-\lambda.
\end{equation*}
\end{lem}

\begin{proof}
Similar to the proof of Lemma \ref{lem2}, it is enough to estimate the kernel $\chi
_{E_{1}}(x){m_{\lambda ,\infty }^{\delta _{p},\gamma }}^{\vee }(x).$ Using
the similar argument as that of Lemma \ref{lem2}, we break up $E_{1}$ into
the regions $\{x\in E_{1}:|x|<1\}$ and $\{x\in E_{1}:|x|\geq 1\}$.

For $|x|<1$, recall that
\begin{equation*}
K_{\lambda ,\infty }^{\delta _{p},\gamma }(x)=\int_{%
\mathbb{R}^{n}}\frac{\Psi _{\infty }\left( {\xi }\right) }{\left\vert \xi
\right\vert ^{\lambda }}e^{2\pi ix\cdot \xi }d\xi .
\end{equation*}
 Let $\eta $ be the function as in Lemma \ref{lem2}. Write $K_{\lambda
,\infty }^{\delta _{p},\gamma }(x)$ as
\begin{equation*}
\int_{\mathbb{R}^{n}}\eta (|x|\xi )\frac{\Psi _{\infty }\left( {\xi }\right)
}{\left\vert \xi \right\vert ^{\lambda }}e^{2\pi ix\cdot \xi }d\xi +\int_{%
\mathbb{R}^{n}}\left( 1-\eta (|x|\xi )\right) \frac{\Psi _{\infty }\left( {%
\xi }\right) }{\left\vert \xi \right\vert ^{\lambda }}e^{2\pi ix\cdot \xi
}d\xi :=I(x)+II(x).
\end{equation*}%
An easy calculation by using Leibniz's rule leads to
\begin{align*}
\partial _{x}^{\beta }I(x)& =\int_{\mathbb{R}^{n}}\frac{\Psi _{\infty }\left( {\xi }\right) }{\left\vert \xi
\right\vert ^{\lambda }}\left(\sum_{\mu+\nu=\beta}\frac{\beta!}{\mu!\text{}\nu!}\partial^{\mu}_{x}\left(\eta (|x|\xi )\right)\partial^{\nu}_{x}\left(e^{2\pi ix\cdot \xi }\right)\right)d\xi   \notag  \label{i1} \\
& =\int_{\mathbb{R}^{n}}\frac{\Psi _{\infty }\left( {\xi }\right) }{\left\vert \xi
\right\vert ^{\lambda }}\left(\sum_{\mu+\nu=\beta}\frac{\beta!}{\mu!\text{}\nu!}(2\pi i \xi)^{\nu}\partial^{\mu}_{x}\left(\eta (|x|\xi )\right)\right)e^{2\pi ix\cdot \xi }d\xi,
\end{align*}
and
\begin{equation*}\label{eta1}
\partial^{\mu}_{x}\left(\eta (|x|\xi )\right) =\sum_{|\alpha|+|\tau|=|\mu|}P_{\tau}(x)
\left(\partial^{\alpha}\eta\right)(|x|\xi)\xi^{\alpha},
\end{equation*}
where $\mu $, $\nu $, $\alpha$ and $\tau$ are $n$-tuple indices,
 and $P_{\tau }(x)\in C^{\infty}(\mathbb{R}^n\backslash\{0\})$  satisfies
\begin{equation}
|P_{\tau }(x)|\preceq |x|^{-|\tau |}.  \label{P}
\end{equation}%

We divide $\mu$ into two cases: $\mu=0$ and $\mu\neq 0$.

Firstly, if $\mu=0$, by the choice of $\eta $ and the support condition for the function $\Psi_{\infty} $,
then we have
\begin{equation*}
\left|\partial^{\beta}I(x)\right|\preceq \int_{1/{|x|}<|\xi |<2/{|x|}}|\xi |^{|\beta |-\lambda }d\xi \preceq
|x|^{-n-|\beta |+\lambda }.
\end{equation*}%

Secondly, when $\mu\neq0$, note that the function $\partial ^{\nu }\eta$ is supported in the set $%
\{x\in \mathbb{R}^{n}:1<|x|<2\}$. Using the condition \eqref{P}, we can see that the term $|\partial^{\beta}I(x)|$
is not greater than a constant multiple of
\begin{align*}
& \quad\sum_{\mu +\nu =\beta }\sum_{|\alpha|+|\tau|=|\mu|}
\frac{1}{|x|^{|\tau |}}\int_{1/{|x|}<|\xi |<2/{|x|}}{%
|\xi |}^{|\nu |+|\alpha|-\lambda }d\xi  \\
& \preceq |x|^{-|\beta |-n+\lambda }.
\end{align*}%

We now estimate the second term  $II(x)$. Without of loss of generality, assume  $x_1\neq 0$ and $0<\lambda<1$,
integration by parts on the variable $\xi _{1}$ for $n-1$ times,
\begin{equation*}
II(x) \approx\frac{1}{{(-2 \pi i|x_1|)}^{n-1}}\int_{
\mathbb{R}^{n}}\partial^{n-1}_{1}\left(\left( 1-\eta (|x|\xi )\right) \frac{\Psi _{\infty }\left( {%
\xi }\right) }{\left\vert \xi \right\vert ^{\lambda }}\right)
e^{2\pi ix\cdot \xi
}d\xi.
\end{equation*}
Using the generalized Leibniz rule, we have
{\allowdisplaybreaks\begin{align*}
&\quad\quad \partial^{n-1}_{1}\left(\left( 1-\eta (|x|\xi )\right) \frac{\Psi _{\infty }\left( {%
\xi }\right) }{\left\vert \xi \right\vert ^{\lambda }}\right)\\
&=\sum_{k_1+k_2+k_3=n-1}C_{k_1,k_2,k_3,n-1}\partial_1^{k_1}(1-\eta(|x|\xi))
  \partial_1^{k_2}(\Psi_{\infty}(\xi))\partial_1^{k_3}(|\xi|^{-\lambda})\\
&=\sum_{k_1+k_2+k_3=n-1,k_1\neq 0}C_{k_1,k_2,k_3,n-1}\partial_1^{k_1}(1-\eta(|x|\xi))
  \partial_1^{k_2}(\Psi_{\infty}(\xi))\partial_1^{k_3}(|\xi|^{-\lambda})\\
  &\quad+\sum_{k_2+k_3=n-1,k_2\neq 0}C_{k_2,k_3,n-1}\left( 1-\eta (|x|\xi )\right)
  \partial_1^{k_2}(\Psi_{\infty}(\xi))\partial_1^{k_3}(|\xi|^{-\lambda})\\
 &\quad+\left( 1-\eta (|x|\xi )\right) \Psi_{\infty}(\xi)\partial_1^{n-1}(|\xi|^{-\lambda})\\
 & =
\sum_{k_1+k_2+k_3=n-1,k_1\neq 0}C_{k_1,k_2,k_3,n-1} \left(\sum_{k_{11}+k_{12}=k_1,k_{11}\neq 0 }C_{k_{11},k_{12},k_1}
  |x|^{k_{11}}(\partial_1^{k_{11}}\eta)(|x|\xi)P_{k_{12}}(\xi)\right)\\
&\quad\quad \times \left(\sum_{k_{21}+k_{22}=k_2}C_{k_{21},k_{22},k_2}(\partial_1^{k_{21}}
\Psi_{\infty})(\xi)P_{k_{22}}(\xi)\right)\left(\sum_{j=0}^{\left[{k_3}/{2}\right]}
C_{j}\frac{\xi_1^{k_3-2j}}{|\xi|^{\lambda+2(k_3-j)}}\right)\\
&\quad+\sum_{k_2+k_3=n-1,k_2\neq 0}C_{k_2,k_3,n-1}\left( 1-\eta (|x|\xi )\right)\\
&\quad\quad\times
  \left(\sum_{k_{21}+k_{22}=k_2,k_{21}\neq 0}C_{k_{21},k_{22},k_2}(\partial_1^{k_{21}}
\Psi_{\infty})(\xi)P_{k_{22}}(\xi)\right)\left(\sum_{j=0}^{\left[{k_3}/{2}\right]}C_{j}
\frac{\xi_1^{k_3-2j}}{|\xi|^{\lambda+2(k_3-j)}}\right)\\
&\quad+\left( 1-\eta (|x|\xi )\right) {\Psi _{\infty }\left( {
\xi }\right) } \sum_{j=0}^{\left[{(n-1)}/{2}\right]}C_{j}\frac{\xi_1^{n-1-2j}}{|\xi|^{\lambda+2(n-1-j)}}\\
&:=II_1(x,\xi)+II_2(x,\xi)+II_3(x,\xi),
 \end{align*}}where $k_1,k_2,k_3,k_{11},k_{12},k_{21},k_{22} $ are nonnegative integers,
and $P_{k_{j2}}(\xi)$ are $ C^{\infty}(\mathbb{R}^n\backslash\{0\})$  functions with
\begin{equation*}
|P_{k_{j2}}(\xi)|\preceq |\xi|^{-k_{j2}},  \label{Pgg1}
\end{equation*} for $j=1, 2$,
and $C_{k_1,k_2,k_3,n-1},C_{k_{11},k_{12},k_1},C_{k_{21},k_{22},k_2}, C_{j} $ are constants which are dependent of the corresponding parameters of subindices.


To obtain the estimate of the derivative of $K_{\lambda,\infty }^{\delta
_{p},\gamma }(x)$, $\partial _{x}^{\beta }K_{\lambda,\infty }^{\delta
_{p},\gamma }(x)$, it suffices to consider the following three terms:
\begin{equation}\label{d111}
\partial_{x}^{\beta}\left(x_1^{1-n}\left(1-\eta(|x|\xi)\right)e^{2\pi i x\cdot\xi}\right),
\end{equation}
\begin{equation}\label{d13}
\partial_{x}^{\beta}\left(x_1^{1-n}|x|^{k}
\left(\partial_1^{k}\eta\right)(|x|\xi)e^{2\pi i x\cdot\xi}\right),\ k\in \mathbb{N}\backslash\{0\},
\end{equation}
\begin{equation}\label{d16}\partial^{\beta}_x\left(\frac{1}{(-2\pi i |x_1|)^{n-1}} \int_{\mathbb{R}^n}\left( 1-\eta (|x|\xi )\right) {\Psi _{\infty }\left( {
\xi }\right) } \frac{\xi_1^{n-1-2j}}{|\xi|^{\lambda+2(n-1-j)}}e^{2\pi i x\cdot\xi} d\xi\right),\  0\leq j\leq \left[\frac{(n-1)}{2}\right].
\end{equation}

We first estimate \eqref{d111}. By an induction similar to \eqref{eta1} we can show that
\begin{equation}\label{eta2}
\partial^{\mu}_{x}\left(1-\eta (|x|\xi )\right) =\sum_{|\alpha|+|\tau|=|\mu|}P_{\tau}(x)
\left(\partial^{\alpha}\eta\right)(|x|\xi)\xi^{\alpha},
\end{equation}
where $\mu $, $\nu $, $\alpha$ and $\tau$ are $n$-tuple indices, $P_{\tau }(x)$ belongs to $C^{\infty}(\mathbb{R}^n\backslash\{0\})$  satisfying \eqref{P}. Applying the generalized Leibniz rule together with \eqref{eta2}, we have
{\allowdisplaybreaks\begin{eqnarray*}\label{d7111}
&&\quad\partial_{x}^{\beta}\left(x_1^{1-n}\left(1-\eta(|x|\xi)\right)e^{2\pi i x\cdot\xi}\right)\\
&=&\sum_{\zeta+\mu +\kappa=\beta}\widetilde{C_{\zeta,\mu,\kappa,\beta}}\partial^{\zeta}_x(x_1^{1-n})
\partial^{\mu}_x\left(1-\eta(|x|\xi)\right)\partial^{\kappa}_x (e^{2\pi i x\cdot\xi})\\
&=&\sum_{\zeta+\mu +\kappa=\beta}{C_{\zeta,\mu,\kappa,\beta}}x_1^{1-n-|\zeta|}
\partial^{\mu}_x\left(1-\eta(|x|\xi)\right){(2 \pi i \xi)}^{\kappa}e^{2\pi i x\cdot\xi}\\
&=&\sum_{\zeta+\mu +\kappa=\beta}{C_{\zeta,\mu,\kappa,\beta}}x_1^{1-n-|\zeta|}
\left(\sum_{|\alpha|+|\tau|=|\mu|}P_{\tau}(x)
\left(\partial^{\alpha}\eta\right)(|x|\xi)\xi^{\alpha}\right){(2 \pi i \xi)}^{\kappa}e^{2\pi i x\cdot\xi},
\end{eqnarray*}}
and then
\begin{eqnarray*}\label{d8111}
&&\quad\left|\chi_{E_1}(x)\partial_{x}^{\beta}\left(x_1^{1-n}\left(1-\eta(|x|\xi)\right)e^{2\pi i x\cdot\xi}\right)\right|\\
&\preceq &\sum_{\zeta+\mu +\kappa=\beta}{C_{\zeta,\mu,\kappa,\beta}}\sum_{|\alpha|+|\tau|=|\mu|} |x|^{1-n-|\zeta|-|\tau|}
\left|\left(\partial^{\alpha}\eta\right)(|x|\xi)\right||\xi|^{|\alpha|+|\kappa|},
\end{eqnarray*} where $\widetilde{C_{\zeta,\mu,\kappa,\beta}}$ and $C_{\zeta,\mu,\kappa,\beta}$ are positive constants depending on $\zeta,\ \mu,\ \kappa$ and $\beta$.

In order to use the different properties of $\eta$ and $\partial^{\alpha}\eta$
$(|\alpha|\neq 0)$, we need to
consider two cases: $\mu\neq 0$ and $\mu=0$.  If $\mu\neq 0$,
together with these estimates above, then we see that
{\allowdisplaybreaks\begin{align*}
&\quad\left|\chi_{E_1}(x)\partial^{\beta}_x\left(\frac{1}{(-2\pi i |x_1|)^{n-1}}\int_{\mathbb{R}^n}II_2(x,\xi)e^{2\pi i x\cdot\xi}\right)d\xi\right|\\
&\preceq  \sum_{\zeta+\mu +\kappa=\beta}{C_{\zeta,\mu,\kappa,\beta}}\sum_{|\alpha|+|\tau|=|\mu|} |x|^{1-n-|\zeta|-|\tau|}
\int_{1/{|x|}<|\xi|<2/{|x|}} \left|\left(\partial^{\alpha}\eta\right)(|x|\xi)\right||\xi|^{|\alpha|+|\kappa|-\lambda-n} d\xi
\\
&\preceq   \sum_{\zeta+\mu +\kappa=\beta}{C_{\zeta,\mu,\kappa,\beta}}\sum_{|\alpha|+|\tau|=|\mu|} \frac{1}{|x|^{n+|\zeta|+|\tau|+|\alpha|+|\kappa|-\lambda}}\\
&\preceq\frac{1}{|x|^{n+|\beta|-\lambda}}.
\end{align*}}If $\mu=0$, then we note that $k_{21}\neq 0$ implies that
$\supp\partial^{k_{21}}\Psi_{\infty}\subset\{\xi\in \mathbb{R}^n: 1/4<|\xi|<1/2\}$. So the conditions $\supp \eta\subset \{\xi\in\mathbb{R}^n: |\xi|>1/{|x|}\}$ and $\supp \partial^{k_{21}}\Psi_{\infty}\subset\{\xi\in \mathbb{R}^n: 1/4<|\xi|<1/2\}$ imply that
$1/2<|x|<1$. Thus we have
{\allowdisplaybreaks\begin{align*}
&\quad\left|\chi_{E_1}(x)\partial^{\beta}_x\left(\frac{1}{(-2\pi i |x_1|)^{n-1}}\int_{\mathbb{R}^n}II_2(x,\xi)e^{2\pi i x\cdot\xi}d\xi\right)\right|\\
&\preceq  \sum_{\zeta+\mu +\kappa=\beta}{C_{\zeta,\mu,\kappa,\beta}}\sum_{|\alpha|+|\tau|=|\mu|} |x|^{1-n-|\zeta|-|\tau|}
\int_{1/{4}<|\xi|<1/2} \left|1-\eta(|x|\xi)\right|\left|(\partial_1^{k_{21}}
\Psi_{\infty})(\xi)\right| d\xi\\
&\preceq  1.
\end{align*}}

The estimate of \eqref{d13} is similar to that of \eqref{d111}, and we yield that
\begin{equation*}
\left|\chi_{E_1}(x)\partial^{\beta}_x\left(\frac{1}{(-2\pi i |x_1|)^{n-1}}\int_{\mathbb{R}^n}II_1(x,\xi)e^{2\pi i x\cdot\xi}d\xi\right)\right|\preceq\frac{1}{|x|^{n+|\beta|-\lambda}}.
\end{equation*}

We now deal with term \eqref{d16}. In fact, we only need to prove the following two cases:
 for some fixed $j$, $0\leq j\leq \left[{(n-1)}/{2}\right]$,\\
$(i)$  if $|\beta|<\lambda+(n-1)$,
\begin{equation*}\label{case1}
\frac{1}{(-2\pi i |x_1|)^{n-1}}\int_{\mathbb{R}^n}\left( 1-\eta (|x|\xi )\right) {\Psi _{\infty }\left( {
\xi }\right) } \frac{\xi_1^{n-1-2j}}{|\xi|^{\lambda+2(n-1-j)}}\partial^{\beta}_x \left(e^{2\pi i x\cdot\xi}\right) d\xi;
\end{equation*}
$(ii)$  if $|\beta|\geq \lambda+(n-1)$,
\begin{equation*}\label{case2}
\frac{1}{(-2\pi i |x_1|)^{n-1}}\partial^{\beta-\alpha}_x \left(\int_{\mathbb{R}^n}\left( 1-\eta (|x|\xi )\right) {\Psi _{\infty }\left( {
\xi }\right) } \frac{\xi_1^{n-1-2j}}{|\xi|^{\lambda+2(n-1-j)}}\partial^{\alpha}_x\left(e^{2\pi i x\cdot\xi}\right) d\xi\right), \ \end{equation*} where $|\alpha|=n-1$.

If $|\beta|<\lambda+(n-1)$, then we have
{\allowdisplaybreaks\begin{align}\label{k3}
\notag &\quad\frac{1}{(2\pi i |x_1|)^{n-1}}\int_{\mathbb{R}^n}\left( 1-\eta (|x|\xi )\right) {\Psi _{\infty }\left( {
\xi }\right) } \frac{\xi_1^{n-1-2j}}{|\xi|^{\lambda+2(n-1-j)}}\partial^{\beta}_x \left(e^{2\pi i x\cdot\xi}\right) d\xi\\
&=\frac{1}{(2\pi i |x_1|)^{n-1}}\int_{\mathbb{R}^n}\left( 1-\eta (|x|\xi )\right) {\Psi _{\infty }\left( {
\xi }\right) } \frac{\xi_1^{n-1-2j}}{|\xi|^{\lambda+2(n-1-j)}}{(2\pi i \xi)}^{\beta}e^{2\pi i x\cdot\xi} d\xi.
\end{align}}Integrating  by parts on the variable $\xi$  $n$ times, term \eqref{k3} is equal to
\begin{equation*}\label{k31}\frac{1}{(-2\pi i |x_1|)^{2n-1}}\int_{\mathbb{R}^n}\partial^{n-1}_{\xi_1}\left(\left( 1-\eta (|x|\xi )\right) {\Psi _{\infty }\left( {
\xi }\right) } \frac{\xi_1^{n-1-2j}}{|\xi|^{\lambda+2(n-1-j)}}\xi^{\beta}\right) e^{2\pi i x\cdot\xi}d\xi.
\end{equation*} Repeating the previous argument, we can see that it is enough to estimate
\begin{equation*}\label{k31}\frac{1}{(2\pi i |x_1|)^{2n-1}}\int_{\mathbb{R}^n}\left( 1-\eta (|x|\xi )\right) {\Psi _{\infty }\left( {\xi }\right) } P_{\tau'}(\xi) e^{2\pi i x\cdot\xi}d\xi,
\end{equation*} where $\tau'$ is an $n$-tuple indices, and $P_{\tau '}(x)\in C^{\infty}(\mathbb{R}^n\backslash\{0\})$  satisfies $$|P_{\tau'}(\xi)|\preceq \frac{1}{{|\xi|}^{\lambda+2n-|\beta|-1}}. $$
Then we conclude
\allowdisplaybreaks{\begin{align*}\label{k41}
&\quad\left|\chi_{E_1}(x)\frac{1}{(2\pi i |x_1|)^{2n-1}}\int_{\mathbb{R}^n}\left( 1-\eta (|x|\xi )\right) {\Psi _{\infty }\left( {\xi }\right) } P_{\tau'}(\xi) e^{2\pi i x\cdot\xi}d\xi\right|\\
&\preceq \frac{1}{|x|^{2n-1}}\int_{|\xi|>1/{|x|}}\frac{1}{{|\xi|}^{\lambda+2n-|\beta|-1}} d\xi
\preceq \frac{1}{|x|^{n+|\beta|-\lambda}}.
\end{align*}}

 If $|\beta|\geq \lambda+(n-1)$, then we have
\begin{eqnarray*}\label{bad1}
&&\quad\frac{1}{(2\pi i |x_1|)^{n-1}}\partial^{\beta-\alpha}_x \left(\int_{\mathbb{R}^n}\left( 1-\eta (|x|\xi )\right) {\Psi _{\infty }\left( {\xi }\right) } \frac{\xi_1^{n-1-2j}}{|\xi|^{\lambda+2(n-1-j)}}\partial^{\alpha}_x\left(e^{2\pi i x\cdot\xi}\right) d\xi\right)\\
&=& \frac{1}{(2\pi i |x_1|)^{n-1}}\partial^{\beta-\alpha}_x \left(\int_{\mathbb{R}^n}\left( 1-\eta (|x|\xi )\right) {\Psi _{\infty }\left( {\xi }\right) } \frac{\xi_1^{n-1-2j}}{|\xi|^{\lambda+2(n-1-j)}}(2\pi i \xi)^{\alpha}e^{2\pi i x\cdot\xi} d\xi\right).
\end{eqnarray*} Note that $|\alpha|=n-1$. For the integral in the bracket above, by performing integration by parts $n$ times on the variable $\xi_1$, we yield that
\begin{align*}
&\quad\int_{\mathbb{R}^n}\left( 1-\eta (|x|\xi )\right) {\Psi _{\infty }\left( {\xi }\right) } \frac{\xi_1^{n-1-2j}}{|\xi|^{\lambda+2(n-1-j)}}(2\pi i \xi)^{\alpha}e^{2\pi i x\cdot\xi} d\xi\\
&=\frac{1}{(2\pi i |x_1|)^{n}}\int_{\mathbb{R}^n}\partial^{n}_{\xi_1}\left(\left( 1-\eta (|x|\xi )\right) {\Psi _{\infty }\left( {\xi }\right) } \frac{\xi_1^{n-1-2j}}{|\xi|^{\lambda+2(n-1-j)}}(2\pi i \xi)^{\alpha}\right) e^{2\pi i x\cdot\xi} d\xi.
\end{align*}
Repeating the previous estimates, we can obtain
$$\left|\frac{\chi_{E_1}(x)}{(2\pi i |x_1|)^{n-1}}\partial^{\beta-\alpha}_x \left(\int_{\mathbb{R}^n}\left( 1-\eta (|x|\xi )\right) {\Psi _{\infty }\left( {\xi }\right) } \frac{\xi_1^{n-1-2j}}{|\xi|^{\lambda+2(n-1-j)}}\partial^{\alpha}_x\left(e^{2\pi i x\cdot\xi}\right) d\xi\right)\right|\preceq \frac{1}{|x|^{n+|\beta|-\lambda}},$$
and then
\begin{equation*}
\left|\chi_{E_1}(x)\partial^{\beta}_x\left(\frac{1}{(-2\pi i |x_1|)^{n-1}}\int_{\mathbb{R}^n}II_3(x,\xi)e^{2\pi i x\cdot\xi}d\xi\right)\right|\preceq\frac{1}{|x|^{n+|\beta|-\lambda}}.
\end{equation*}

Therefore, we show that for $|x|<1$,
$$|\partial_{x}^{\beta}K_{\lambda,\infty}^{\delta_p,\gamma}(x)|\preceq \frac{1%
}{|x|^{n+|\beta|-\lambda}}.$$  In addition, if a nonnegative integer $k\leq n$ and
 $k\leq \lambda<k+1$, then we only need to take integration by parts on the variable $\xi$ for $n-{k+1}$ times. If $k>n$, it is a triviality.
For the case $|x|\geq 1$, recall that 
\begin{equation*}
K_{\lambda ,\infty }^{\delta _{p},\gamma }(x)=\int_{%
\mathbb{R}^{n}}\frac{\Psi _{\infty }\left( {\xi }\right) }{\left\vert \xi
\right\vert ^{\lambda }}e^{2\pi ix\cdot \xi }d\xi .
\end{equation*}%

Integration by parts on the variable $\xi_{1}$ for $n+L$ times yields that
\begin{equation*}
K_{\lambda ,\infty }^{\delta
_{p},\gamma }(x) =
  \frac{1}{{(-2 \pi i|x_1|)}^{n+L}}\int_{\mathbb{R}^n}\partial^{n+L}_{\xi_1}\left(\frac{\Psi_{\infty}(\xi)}{|\xi|^{\lambda}}\right) e^{2\pi i x\cdot \xi} d\xi.
 \end{equation*}
  Repeated calculations as before enables us to obtain that
{\allowdisplaybreaks\begin{align*}
  & \partial^{n+L}_{\xi_1}\left(\frac{\Psi_{\infty}(\xi)}{|\xi|^{\lambda}}\right)=\sum_{k_1+k_2=n+L}C_{k_{1},k_{2},n+L}\\
  &\times \left(\sum_{k_{11}+k_{12}=k_1, k_{11}\neq 0} C_{k_{11},k_{12},k_1}(\partial_1^{k_{11}}
\Psi_{\infty})(\xi)P_{k_{12}}(\xi)\right)\left(\sum_{j=0}^{\left[{k_2}/{2}\right]}
\frac{C_{j} \xi_1^{k_2-2j}}{|\xi|^{\lambda+2(k_2-j)}}\right)\\
&\quad +\Psi_{\infty}(\xi)\left(\sum_{j=0}^{\left[{(n+L)}/{2}\right]}
\frac{C_{j} \xi_1^{n+L-2j}}{|\xi|^{\lambda+2(n+L-j)}}\right),
\end{align*}}%
where $C_{k_{1},k_{2},n+L},\  C_{k_{11},k_{12},k_1},  C_{j} $ are constants, and $P_{k_{12}}(\xi)\in C^{\infty}(\mathbb{R}^n\backslash\{0\})$
satisfies
\begin{equation}
|P_{k_{12}}(\xi)|\preceq |\xi|^{-k_{12}}. \label{Pgg}
\end{equation}

To express clearly the estimate of the derivative of $\partial _{x}^{\beta }K_{\lambda,\infty }^{\delta
_{p},\gamma }(x)$, we need to introduce some more notation. For any $n$-tuple $\beta=(\beta_1,\beta_2,\ldots,\beta_n)$,
set $\beta'=(\beta_2,\ldots,\beta_n)$, then $\beta=(\beta_1,\beta')$.
Similarly, for any $n$-tuple $x=(x_1, x_2, \ldots,x_n)$ write $x=(x_1,x')$, where $x'=(x_2,\ldots,x_n)$.
Now for $|x|\geq 1$, we use Leibniz's rule to obtain \begin{align*}
&\quad \partial_{x_1}^{\beta_1}\partial_{x'}^{\beta'}\left(\left(-2\pi i|x_1|\right)^{-(n+L)} e^{2\pi i x\cdot \xi}\right)\\
& =\partial_{x_1}^{\beta_1}
\left((2\pi ix_1)^{-(n+L)}e^{2 \pi i x_1\xi_1}\right)\partial^{\beta'}_{x'}\left(e^{2\pi ix'\cdot
\xi' }\right) \\
& =C_{\beta'}{\xi'}^{\beta'}e^{2\pi ix\cdot
\xi } \sum_{0\leq l\leq \beta_1}C_{l}x_1^{-(n+L+l)}\xi_1^{\beta_1-l}.
\end{align*}%
As a result, invoking the estimate of \eqref{Pgg} and the properties of $\Psi_{\infty}$, we yield that
{\allowdisplaybreaks \begin{align*}
&\quad\left|\chi _{E_{1}}(x)\partial _{x}^{\beta }K_{\lambda,\infty }^{\delta
_{p},\gamma }(x)\right|\\
&\preceq \frac{1}{|x|^{n+L}} \sum_{k_1+k_2=n+L} \sum_{k_{11}+k_{12}=k_1, k_{11}\neq 0} \sum_{j=0}^{\left[{(n+L)}/{2}\right]}
\int_{1/4<|\xi|<1/2}\frac{1}{|\xi|^{\lambda+{k_2}-|\beta|+j+k_{12}}}d\xi\\
&\quad+\frac{1}{|x|^{n+L}} \sum_{j=0}^{\left[{(n+L)}/{2}\right]}
\sum_{0\leq l\leq \beta_1} \frac{1}{|x|^l}\left|\int_{\mathbb{R}^{n}}\frac{%
\Psi _{\infty }(\xi )}{{|\xi |}^{\lambda+2(n+L-j) }}{\xi'}^{\beta'} \xi_1^{n+L-2j+\beta_1-l}e^{2\pi ix\cdot
\xi } d\xi\right|\\
&\preceq \frac{1}{|x|^{n+L}} \sum_{0\leq l\leq \beta_1} \frac{1}{|x|^l}\int_{|\xi|>1/4}\frac{%
1}{{|\xi |}^{n+\lambda+L+l-|\beta| }} d\xi\\
&\quad+\frac{1}{|x|^{n+L}} \sum_{0\leq l\leq \beta_1} \frac{1}{|x|^l}\int_{|\xi|>1/4}\frac{%
1}{{|\xi |}^{n+\lambda+L+l-|\beta| }} d\xi\\
&\preceq \frac{1}{|x|^{n+L}},
\end{align*}}%
because of $L>|\beta |-\lambda $ and $|x|>1$. Similarly, we can estimate the other terms
$ \chi _{E_{j}}(x)\partial _{x}^{\beta }K_{\lambda,\infty }^{\delta
_{p},\gamma }(x)$ for $2\leq j\leq n$. Combining all the estimates, we complete
the proof of this lemma.
\end{proof}

\section{Proof of Theorems \protect\ref{t3} and \protect\ref{t2}}

\label{s3} We first give the proof of Theorem \ref{t2}. We shall present
some important lemmas.

\begin{lem}[\protect\cite{S2}]
\label{lem4} Let $0<p\leq 1$. Suppose that a function $\zeta$ vanishes at $%
\infty$ and satisfies
\begin{equation*}
\left|\partial_x^{\beta}\zeta(x)\right|\preceq \left(1+|x|\right)^{-A}
\end{equation*}
for $|\beta|=[n(1/p-1)]+1$ and $A>n/p$. Then there is a constant $C>0$, for
any $f\in H^p(\mathbb{R}^n)$,
\begin{equation*}
\left\|\sup_{R>0}|f\ast\zeta_{R}|\right\|_{L^p(\mathbb{R}^n)}\leq
C\|f\|_{H^p(\mathbb{R}^n)},
\end{equation*}%
where $\zeta_{R}(x)=R^{-n}\zeta(x/R)$ for $R>0$ and $x\in \mathbb{R}^n$.
\end{lem}

\begin{lem}[\protect\cite{FZ}]
\label{lem1} Let $1<q<\infty $ and $\delta >(n-1)|1/q-1/2|$. For any $%
\lambda \geq 0$, we have
\begin{equation*}
\left\Vert \sup_{R>0}\left|T_{R, \lambda ,1}^{\delta ,\gamma
}g\right|\right\Vert _{L^{q}(\mathbb{R}^{n})}\preceq \Vert g\Vert _{L^{q}(%
\mathbb{R}^{n})}.
\end{equation*}
\end{lem}

\begin{lem}[\protect\cite{FZ}]
\label{lem111} Let $1<q<\infty $ and $\delta >-1$. For any $\lambda \geq 0$,
we have
\begin{equation*}
\left\Vert \sup_{R>0}\left|T_{R, \lambda ,\infty}^{\delta ,\gamma
}g\right|\right\Vert _{L^{q}(\mathbb{R}^{n})}\preceq \Vert g\Vert _{L^{q}(%
\mathbb{R}^{n})}.
\end{equation*}
\end{lem}

\medskip

\noindent \textit{Proof of Theorem \ref{t2}.} As mentioned above,  we
know that the distribution
\begin{equation*}
R^{\lambda }\{S_{R}^{\delta _{p},\gamma }(f)-f\}
\end{equation*}%
has the Fourier transform $\ \mu (\cdot /R)$ $\widehat{g},$ where $%
g=I_{-\lambda }f$ and the multiplier $\mu $  is a radial function given
by
\begin{equation*}
\mu (\xi )=\frac{(1-\left\vert \xi \right\vert ^{\gamma })_{+}^{\delta_p }-1}{%
\left\vert \xi \right\vert ^{\lambda }},\ \xi \neq 0\text{ \ and \ }\mu
(0)=\lim_{t\rightarrow 0^{+}}\frac{(1-t^{\gamma })_{+}^{\delta_p }-1}{%
t^{\lambda }}.
\end{equation*}

Since  $\mu (\xi )=m_{\lambda ,0}^{\delta _{p},\gamma }(\xi )+m_{\lambda
,1}^{\delta _{p},\gamma }(\xi )+m_{\lambda ,\infty }^{\delta _{p},\gamma
}(\xi ),$ to prove
\begin{equation*}
\left\vert \left\{ x\in \mathbb{R}^{n}:({M}_{\lambda }^{\delta _{p},\gamma
}f)(x)>s\right\} \right\vert \preceq \left( \frac{\Vert f\Vert _{I_{\lambda
}(H^{p})(\mathbb{R}^{n})}}{s}\right) ^{p},
\end{equation*}%
it suffices to show that the maximal operators $\sup_{R>0}|T_{R,\lambda
,j}^{\delta _{p},\gamma }g|$ satisfy for any $g\in H^{p}(\mathbb{R}^{n})$,
\begin{equation*}
\left\vert \left\{ x\in \mathbb{R}^{n}:\sup_{R>0}\left\vert T_{R,\lambda
,j}^{\delta _{p},\gamma }g(x)\right\vert >s\right\} \right\vert \preceq
\left( \frac{\Vert g\Vert _{H^{p}(\mathbb{R}^{n})}}{s}\right) ^{p},\
j=0,1,\infty .  \label{Tg1}
\end{equation*}

We first consider the case $j=0$. For any multi-index $\beta $  satisfying $%
|\beta |=[n(1/p-1)]+1$, then
\begin{equation*}
n+|\beta |+\gamma -\lambda >n/p+\gamma -\lambda \geq n/p.
\end{equation*}%
It follows from Lemmas \ref{lem2} and \ref{lem4} that
\begin{equation*}
\left\Vert \sup_{R>0}\left\vert T_{R,\lambda ,0}^{\delta _{p},\gamma }\left(
g\right) \right\vert \right\Vert _{L^{p}(\mathbb{R}^{n})}\preceq \Vert
g\Vert _{H^{p}(\mathbb{R}^{n})}.
\end{equation*}%
Furthermore
\begin{align*}
\left\vert \left\{ x\in \mathbb{R}^{n}:\sup_{R>0}\left\vert T_{R,\lambda
,0}^{\delta _{p},\gamma }g(x)\right\vert >s\right\} \right\vert & \preceq
\int_{\mathbb{R}^{n}}\left( \frac{\sup_{R>0}\left\vert T_{R,\lambda
,0}^{\delta _{p},\gamma }g(x)\right\vert }{s}\right) ^{p}dx \\
& \preceq \left( {\Vert g\Vert _{H^{p}(\mathbb{R}^{n})}}{s}^{-1}\right) ^{p}.
\end{align*}

Similar to \cite{STW}, using a summation formula of Stein and N. Weiss  \cite{SW2},
in order to prove the above estimate \eqref{Tg1}, it suffices to show that
for any $(p,2)$-atom $b$
\begin{equation*}
\left\vert \left\{ x\in \mathbb{R}^{n}:\sup_{R>0}\left\vert T_{R,\lambda
,j}^{\delta _{p},\gamma }b(x)\right\vert >s\right\} \right\vert \preceq
s^{-p},\ j=1,\infty .  \label{Tg}
\end{equation*}%
Without loss of generality, we can assume the atom $b$ supported in a cube $Q(0,l)$ centered at the origin with
diameter $l>0$ whose sides parallel to the axes (see also Grafakos \cite[p.94]{G2}). Let $Q^{\ast }=2\sqrt{n}Q$ denote the concentric cube
with $2\sqrt{n}$-times the diameter of $Q$.

Consider the case $j=1$. We write
\begin{align*}
&\quad\left|\left\{x\in \mathbb{R}^n: \sup_{R>0} \left|T_{R,\lambda ,1}
^{\delta_p ,\gamma }b(x)\right|>s \right\}\right| \\
&=\left|\left\{x\in Q^{*}: \sup_{R>0} \left|T_{R,\lambda ,1} ^{\delta_p
,\gamma }b(x)\right|>s \right\}\right| + \left|\left\{x\in \mathbb{R}%
^n\backslash Q^{*}: \sup_{R>0} \left|T_{R,\lambda ,1} ^{\delta_p ,\gamma
}b(x)\right|>s \right\}\right|.
\end{align*}

Note that $\sup_{R>0} \left|T_{R,\lambda ,1} ^{\delta_p ,\gamma }b(x)\right|$
is bounded on $L^2(\mathbb{R}^n)$ since $\delta_{p}$ satisfies the
assumption of Lemma \ref{lem1}. Make use of H\"{o}lder's inequality to
obtain the first term
\begin{align*}
\left|\left\{x\in Q^{*}: \sup_{R>0} \left|T_{R,\lambda ,1} ^{\delta_p
,\gamma }b(x)\right|>s \right\}\right| &\leq s^{-p}\int_{Q^{*}} \sup_{R>0}
\left|T_{R,\lambda ,1} ^{\delta_p ,\gamma }b(x)\right|^p dx \\
&\preceq s^{-p} |Q|^{1-p/2} \left\|\sup_{R>0} \left|T_{R,\lambda ,1}
^{\delta_p ,\gamma }b\right| \right\|_{L^2(\mathbb{R}^n)}^p \\
&\preceq s^{-p} |Q|^{1-p/2} \|b\|_{L^2(\mathbb{R}^n)}^p \\
&\preceq s^{-p},
\end{align*}
where in the last inequality we have used the size condition of $b$.

Next, we turn to estimate the second term
\begin{equation*}
\left|\left\{x\in \mathbb{R}^n\backslash {Q^{*}}: \sup_{R>0}
\left|T_{R,\lambda ,1} ^{\delta_p ,\gamma }b(x)\right|>s \right\}\right|.
\end{equation*}
Divide $R>0$ into two cases: $Rl>2$ and $0<Rl\leq 2$.

\medskip

Case 1: $Rl>2$. \ For $x\in \mathbb{R}^{n}\backslash {Q^{\ast }}$ and $|y-x|<%
\sqrt{n}l/2$, we have $|x-y|<|x|/2$ and $R|y|>R|x|/2>Rl>2$. Apply Lemma \ref%
{lem21} with $\beta =0$, then
\begin{equation*}
\left\vert (K_{\lambda ,1}^{\delta _{p},\gamma })_{1/R}(y)\right\vert
\preceq R^{n}(R|x|)^{-n/p}.
\end{equation*}%
Invoking the above estimate and H\"{o}lder's inequality, we have
\begin{equation*}
\left\vert T_{R,\lambda ,1}^{\delta _{p},\gamma }b(x)\right\vert \preceq
\Vert b\Vert _{L^{2}(\mathbb{R}^{n})}\left( \int_{Q}\left\vert (K_{\lambda
,1}^{\delta _{p},\gamma })_{1/R}(y)\right\vert ^{2}dy\right) ^{1/2}\preceq
|x|^{-n/p}.
\end{equation*}

Case 2: $0<Rl\leq 2$. Let the multi-index $\beta $ satisfy $|\beta
|=[n(1/p-1)]+1$. Lemma \ref{lem21} implies that
\begin{equation}
|\partial _{x}^{\beta }(K_{\lambda ,1}^{\delta _{p},\gamma
})_{1/R}(x)|\preceq R^{n+|\beta |},\text{\ if\ }R|x|\leq 1,  \label{e1}
\end{equation}%
and
\begin{equation}
|\partial _{x}^{\beta }(K_{\lambda ,1}^{\delta _{p},\gamma
})_{1/R}(x)|\preceq \frac{R^{n+|\beta |}}{|Rx|^{n/p}},\text{\ if\ }R|x|>1.
\label{e2}
\end{equation}

By virtue of Taylor's theorem, we see that
\begin{equation*}
(K_{\lambda ,1}^{\delta _{p},\gamma })_{\frac{1}{R}}(x-y)=\sum_{|\alpha |\leq |\beta
|-1}\frac{\left(\partial _{x}^{\alpha }(K_{\lambda ,1}^{\delta _{p},\gamma
})_{\frac{1}{R}}\right)(x)}{\alpha !}y^{\alpha }+\sum_{|\alpha |={|\beta |}}\frac{\left(\partial
_{x}^{\alpha }(K_{\lambda ,1}^{\delta _{p},\gamma })_{\frac{1}{R}}\right)(x-(1-\theta )y)}{%
\alpha !}y^{\alpha },
\end{equation*}%
for some $\theta \in (0,1)$. Using the cancellation condition of $b$, we
have that
\begin{align*}
\quad T_{R,\lambda ,1}^{\delta _{p},\gamma }b(x)& =\int_{Q}b(y)(K_{\lambda
,1}^{\delta _{p},\gamma })_{1/R}(x-y)dy \\
& =\sum_{|\alpha |={|\beta |}}\frac{1}{\alpha !}\int_{Q}b(y)y^{\alpha }\left({
\partial _{x}^{\alpha }(K_{\lambda ,1}^{\delta _{p},\gamma })_{1/R}}\right)
(x-(1-\theta )y)dy,
\end{align*}%
and also
\begin{align*}
\quad \sup_{0<R\leq 2/l}\left\vert T_{R,\lambda ,1}^{\delta _{p},\gamma
}b(x)\right\vert & \preceq \sum_{|\alpha |={|\beta |}}\frac{1}{\alpha !}%
\int_{Q}|b(y)||y|^{|\alpha| }\\
&\quad\times \sup_{0<R\leq 2/l}\left( \sup_{|z-x|<|y|}\left|\left(\partial
^{\alpha }(K_{\lambda ,1}^{\delta _{p},\gamma })_{1/R}\right)(z)\right|\right)dy  \\
& \preceq l^{n+|\beta |-n/p}\cdot \sup_{0<R\leq 2/l}\left(
\sup_{|z-x|<|y|}\left|\left(\partial^{\alpha }(K_{\lambda ,1}^{\delta _{p},\gamma
})_{1/R}\right)(z)\right|\right) .
\end{align*}%

Denote by $U(x)$ the set $\{z\in \mathbb{R}^{n}:|z-x|<|y|,\ y\in Q\ \text{and%
}\ x\in \mathbb{R}^{n}\backslash Q^{\ast }\}$. Observe that the
condition  $z\in U(x)$ implies that ${|x|}/2<|z|<{3|x|}/2$. We need to
divide the set $U(x)$ into two regions: $\{z:|z-x|<|y|\ \text{ and}\
R|z|\leq 1\}$ and $\{z:|z-x|<|y|\ \text{ and}\ R|z|>1\}$. For the first
region, we use the estimate \eqref{e1} to yield
\begin{align*}
& \sup_{0<R\leq 2/l}\left( \sup_{z\in U(x),R|z|\leq 1}
\left|\left(\partial^{\alpha }(K_{\lambda ,1}^{\delta _{p},\gamma
})_{1/R}\right)(z)\right|\right) \\
&\quad\quad\preceq \sup_{0<R\leq 2/l,z\in U(x),R|z|\leq 1}R^{n+|\beta |-n/p}R^{n/p} \\
& \quad\quad\preceq \sup_{0<R\leq 2/l}R^{n+|\beta |-n/p}\cdot \sup_{R|z|\leq 1,z\in
U(x)}R^{n/p} \\
& \quad\quad\preceq l^{-(n+|\beta |-n/p)}\cdot |x|^{-n/p},
\end{align*}%
where we have used $0<n+|\beta |-n/p\leq 1$ and $|z|\simeq |x|$. For the
second region, we use the estimate \eqref{e2} to yield
{\allowdisplaybreaks\begin{align*}
& \sup_{0<R\leq 2/l}\left( \sup_{z\in U(x),R|z|>1}
\left|\left(\partial^{\alpha }(K_{\lambda ,1}^{\delta _{p},\gamma
})_{1/R}\right)(z)\right|\right) \\
&\quad\quad \preceq
\sup_{0<R\leq 2/l,R|z|>1,z\in U(x)}\frac{R^{n+|\beta |}}{|Rz|^{n/p}} \\
&\quad\quad \preceq \sup_{0<R\leq 2/l,z\in U(x)}\left( R^{n+|\beta |-\frac{n}{p}}\cdot
|z|^{-\frac{n}{p}}\right)  \\
& \quad\quad \preceq l^{-(n+|\beta |-n/p)}\cdot |x|^{-n/p}.
\end{align*}}%
Hence, using H\"{o}lder's inequality and the size condition of the $(p,2)$%
-atom $b$, we have
$$\int_{Q}|b(y)|\cdot |y|^{|\alpha| } dy\preceq l^{n+|\beta |-n/p},$$
and of course
\begin{equation*}
\left\vert \left\{ x\in \mathbb{R}^{n}\backslash {Q^{\ast }}%
:\sup_{R>0}\left\vert T_{R,\lambda ,1}^{\delta _{p},\gamma }b(x)\right\vert
>s\right\} \right\vert \preceq \left\vert \left\{ x\in \mathbb{R}%
^{n}\backslash {Q^{\ast }}:|x|^{-n/p}>s\right\} \right\vert \preceq {s^{-p}}.
\end{equation*}

At last, we will deal with the case $j=\infty$. Write
\begin{align*}
&\quad\left|\left\{x\in \mathbb{R}^n: \sup_{R>0} \left|T_{R,\lambda ,\infty}
^{\delta_p ,\gamma }b(x)\right|>s \right\}\right| \\
&=\left|\left\{x\in Q^{*}: \sup_{R>0} \left|T_{R,\lambda,\infty} ^{\delta_p
,\gamma }b(x)\right|>s \right\}\right| + \left|\left\{x\in \mathbb{R}%
^n\backslash {Q^{*}}: \sup_{R>0} \left|T_{R,\lambda ,\infty} ^{\delta_p
,\gamma }b(x)\right|>s \right\}\right|.
\end{align*}

Lemma \ref{lem111} gives that $\sup_{R>0} \left|T_{R,\lambda ,\infty}
^{\delta_p ,\gamma }b(x)\right|$ is bounded on $L^2(\mathbb{R}^n)$. Then the
first term can be estimated by means of H\"{o}lder's inequality and the size
condition of $b$,
\begin{align*}
\left|\left\{x\in Q^{*}: \sup_{R>0} \left|T_{R,\lambda ,\infty} ^{\delta_p
,\gamma }b(x)\right|>s \right\}\right| &\leq s^{-p}\int_{Q^{*}} \sup_{R>0}
\left|T_{R,\lambda ,\infty} ^{\delta_p ,\gamma }b(x)\right|^p dx \\
&\preceq s^{-p} |Q|^{1-p/2} \left\|\sup_{R>0} \left|T_{R,\lambda ,\infty}
^{\delta_p ,\gamma }b\right| \right\|_{L^2(\mathbb{R}^n)}^p \\
&\preceq s^{-p} |Q|^{1-p/2} \|b\|_{L^2(\mathbb{R}^n)}^p \\
&\preceq s^{-p}.
\end{align*}

Next, we turn to estimate the second term
\begin{equation*}
\left|\left\{x\in \mathbb{R}^n\backslash {Q^{*}}: \sup_{R>0}
\left|T_{R,\lambda ,\infty} ^{\delta_p ,\gamma }b(x)\right|>s
\right\}\right|.
\end{equation*}
By the previous proof, it suffices to show that for $x\in \mathbb{R}%
^n\backslash {Q^{*}}$
\begin{equation*}
\sup_{R>0}\left|T_{R,\lambda ,\infty}^{\delta_p ,\gamma }b(x)\right| \preceq
|x|^{-n/p}.
\end{equation*}

Applying the similar method as that of the case   $j=1$, we break $R>0$ into two
cases: $Rl>2$ and $0<Rl\leq 2$. We also take $L=[n(1/p-1)]+1$ for $0<p\leq
1$.

Case 1: $Rl>2$. \ For $x\in \mathbb{R}^{n}\backslash {Q^{\ast }}$ and $%
|y-x|<l$, we have $|x-y|<|x|/2$ and $R|y|>R|x|/2>Rl>2$. Apply Lemma \ref%
{lem3} with $\beta =0$, then
\begin{equation*}
\left\vert (K_{\lambda ,\infty }^{\delta _{p},\gamma })_{1/R}(y)\right\vert
\preceq R^{n}(R|x|)^{-n-[n(1/p-1)]-1}.
\end{equation*}%
Invoking the above estimate and H\"{o}lder's inequality, we have
\begin{equation*}
\left\vert T_{R,\lambda ,\infty }^{\delta _{p},\gamma }b(x)\right\vert
\preceq \Vert b\Vert _{L^{2}(\mathbb{R}^{n})}\left( \int_{Q}\left\vert
(K_{\lambda ,\infty }^{\delta _{p},\gamma })_{1/R}(y)\right\vert
^{2}dy\right) ^{1/2}\preceq |x|^{-n/p}.
\end{equation*}%

Case 2: $0<Rl\leq 2$. For any multi-index $\beta $ such that $%
|\beta |=L$, it follows from Lemma \ref{lem3} that
\begin{equation}
|\partial _{x}^{\beta }(K_{\lambda ,\infty }^{\delta _{p},\gamma
})_{1/R}(x)|\preceq \frac{R^{n+|\beta |}}{|Rx|^{n+|\beta |-\lambda }},\text{%
\ if\ }R|x|\leq 1,  \label{e31}
\end{equation}%
and
\begin{equation}
|\partial _{x}^{\beta }(K_{\lambda ,\infty }^{\delta _{p},\gamma
})_{1/R}(x)|\preceq \frac{R^{n+|\beta |}}{|Rx|^{n+L}},\text{\ if\ }R|x|>1.
\label{e32}
\end{equation}

Applying Taylor's theorem and the cancellation condition of $b$, one has
 \begin{align*}
&\quad \sup_{0<R\leq 2/l}\left\vert T_{R,\lambda ,\infty }^{\delta
_{p},\gamma }b(x)\right\vert \\
& \preceq \sum_{|\alpha |={|\beta |}}\frac{1}{%
\alpha !}\int_{Q}|b(y)|\cdot |y|^{|\alpha| }\sup_{0<R\leq 2/l}\left(
\sup_{|z-x|<|y|}\left|\left(\partial ^{\alpha }(K_{\lambda ,\infty }^{\delta
_{p},\gamma })_{1/R}\right)(z)\right|\right)  dy.
\end{align*}%

In the same manner we break the set $U(x)$ into two regions: $\{z:
|z-x|<|y| \ \text{ and}\ R|z|\leq 1 \}$ and $\{z: |z-x|<|y|\ \text{ and}\
R|z|>1\}$, where $U(x)$ is defined as before. For the first region, we use
the estimate \eqref{e31} to yield
\begin{align*}
& \quad \sup_{0<R\leq 2/l}\left(\sup_{z\in U(x), R|z|\leq 1
}\left|\left(\partial ^{\alpha }(K_{\lambda ,\infty }^{\delta
_{p},\gamma })_{1/R}\right)(z)\right|
\right) \preceq \sup_{0<R\leq 2/l,R|z|\leq 1, z\in U(x)}\frac{R^{n+|\beta|}}{%
|Rz|^{n+|\beta|-\lambda}} \\
& \preceq {|x|}^{-n/p} \cdot l^{-(n+|\beta|-n/p)}\cdot \sup_{z\in U(x)}{%
\left(\frac{l}{|z|}\right)}^{n+|\beta|-n/p} \\
&\leq {|x|}^{-n/p}\cdot l^{-(n+|\beta|-n/p)},
\end{align*}
where we have used $0<n+|\beta|-n/p\leq 1$ and ${3|x|}/2>|z|> {|x|}/2>{\sqrt{%
n}l}/2$. Similarly, we use \eqref{e32} to estimate the supremum in the
second region
\begin{align*}
& \quad \sup_{0<R\leq 2/l}\left(\sup_{z\in U(x), R|z|> 1 }
\left|\left(\partial ^{\alpha }(K_{\lambda ,\infty }^{\delta
_{p},\gamma })_{1/R}\right)(z)\right|
\right) \preceq \sup_{0<R\leq 2/l, z\in U(x), R|z|> 1}\frac{R^{n+|\beta|}}{%
|Rz|^{n+L}} \\
& \preceq |x|^{-n/p} \cdot l^{-(n+|\beta|-n/p)}\cdot \sup_{z\in U(x)}{\left(%
\frac{l}{|z|}\right)}^{n+|\beta|-n/p} \\
& \preceq |x|^{-n/p} \cdot l^{-(n+|\beta|-n/p)}.
\end{align*}
Therefore, H\"{o}lder's inequality and the
size estimate of the $(p,2)$-atom $b$ lead to
\begin{equation*}
\sup_{R>0}\left|T_{R,\lambda ,\infty}^{\delta_p ,\gamma }b(x)\right| \preceq
|x|^{-n/p}.
\end{equation*}

Summing all the estimates, we complete the proof of Theorem \ref{t2}.
\qed

\bigskip

\noindent \textit{Proof of Theorem \ref{t3}.} The proof follows a standard
method. We only prove the case  $\lambda <\gamma $. Suppose $f\in
I_{\lambda }(H^{p})(\mathbb{R}^{n})$. Since $\mathscr{S}(\mathbb{R}^{n})\cap I_{\lambda }(H^{p})(\mathbb{R}^{n})$ is
dense in  $I_{\lambda }(H^{p})(\mathbb{R}^{n}),$ for any  $\varepsilon >0$, we choose an  $%
h\in \mathscr{S}(\mathbb{R}^{n})\cap I_{\lambda }(H^{p})(\mathbb{R}^{n})$  such that
\begin{equation*}
\left\Vert f-h\right\Vert _{I_{\lambda }(H^{p})(\mathbb{R}^{n})}<\varepsilon .
\end{equation*}
Since \
\begin{equation*}
\lim_{R\rightarrow \infty }\left\vert R^{\lambda }\left( (S_{R}^{\delta
_{p},\gamma }h)(x)-h(x)\right) \right\vert =0
\end{equation*}%
for all  $x,$  for any fixed $s>0,$ we have%
{\allowdisplaybreaks\begin{eqnarray*}
&&\left\vert \left\{ x\in \mathbb{R}^{n}:\lim \sup_{R\rightarrow \infty
}\left\vert R^{\lambda }\left( (S_{R}^{\delta _{p},\gamma }f)(x)-f(x)\right)
\right\vert >s\right\} \right\vert  \\
&\leq &\left\vert \left\{ x\in \mathbb{R}^{n}:\lim \sup_{R\rightarrow \infty
}\left\vert R^{\lambda }\left( (S_{R}^{\delta _{p},\gamma }\left( f-h\right)
)(x)-\left( f-h\right) (x)\right) \right\vert >s/2\right\} \right\vert  \\
&\preceq &\left( \frac{\Vert f-h\Vert _{I_{\lambda }(H^{p})(\mathbb{R}^{n})}%
}{s}\right) ^{p}<\left( \frac{\varepsilon }{s}\right) ^{p}.
\end{eqnarray*}}%

To show that
\begin{equation*}
(S_{R}^{\delta _{p},\gamma }f)(x)-f(x)=O(1/{R^{\gamma }})\ a.e.\ as\
R\rightarrow \infty
\end{equation*}%
is sharp when $\lambda =\gamma $, we pick an  $f\in I_{\lambda }(H^{p})(\mathbb{R}^{n})\cap
\mathscr{S}(\mathbb{R}^{n})$  satisfying
\begin{equation*}
\int_{%
\mathbb{R}
^{n}}\widehat{f}(\xi )\left\vert \xi \right\vert ^{\lambda }e^{2\pi i\xi \cdot
y}d\xi \neq 0
\end{equation*}%
at some  $y\in
\mathbb{R}
^{n}.$ Using the continuity, we have that
\begin{equation*}
\int_{%
\mathbb{R}
^{n}}\widehat{f}(\xi )\left\vert \xi \right\vert ^{\lambda }e^{2\pi ix\cdot \xi
}d\xi \neq 0,
\end{equation*}%
in a neighborhood of  $y.$

Write%
\begin{eqnarray*}
R^{\lambda }\left( (S_{R}^{\delta _{p},\gamma }f)(x)-f(x)\right)
&=&R^{\lambda }\int_{%
\mathbb{R}
^{n}}\left( \left( 1-\frac{{|\xi |^{\lambda }}}{R^{\lambda }}\right)
_{+}^{\delta _{p}}-1\right) \widehat{f}(\xi )e^{2\pi ix\cdot \xi }d\xi  \\
&=&R^{\lambda }\int_{\left\vert \xi \right\vert <R/2}\left( \left( 1-\frac{{%
|\xi |^{\lambda }}}{R^{\lambda }}\right) ^{\delta _{p}}-1\right)
\widehat{f}(\xi )e^{2\pi ix\cdot \xi }d\xi  \\
&&+R^{\lambda }\int_{\left\vert \xi \right\vert \geq R/2}\left( \left( 1-%
\frac{{|\xi |^{\lambda }}}{R^{\lambda }}\right) _{+}^{\delta _{p}}-1\right)
\widehat{f}(\xi )e^{ix\cdot \xi }d\xi  \\
&=&I_{R}(x)+II_{R}(x).
\end{eqnarray*}%
Since  $f\in \mathscr{S}(\mathbb{R}^n),$ we have
\begin{equation*}
\left\vert II_{R}(x)\right\vert \preceq R^{\lambda }\int_{\left\vert \xi
\right\vert \geq R/2}\left\vert \widehat{f}(\xi )\right\vert d\xi \preceq
R^{-L},
\end{equation*}%
for any  $L>0$. Using the Taylor expansion, we have
\begin{equation*}
I_{R}(x)=-\delta _{p}\int_{\left\vert \xi \right\vert <R/2}{|\xi |^{\lambda }%
}\widehat{f}(\xi )e^{2\pi ix\cdot \xi }d\xi +O(\frac{1}{R^{\lambda }}).
\end{equation*}%
Thus, we obtain that \
\begin{equation*}
\lim_{R\rightarrow \infty }R^{\lambda }\left( (S_{R}^{\delta _{p},\gamma
}f)(x)-f(x)\right) \neq 0
\end{equation*}%
in a set of positive measure. This concludes the proof of the theorem.
\qed

\section{Proof of Theorem \protect\ref{t1}}

\label{s4} We need to define a slight refinement of the maximal
function by the form
\begin{equation*}
(\mathcal{M}_{\lambda }^{\delta ,\gamma }f)(x)=\sup_{R>1}|R^{\lambda+\delta-%
\frac{n-1}{2} }\{(S_{R}^{\delta ,\gamma }f)(x)-f(x)\}|.
\end{equation*}%
Theorem \ref{t1} will be proved partially based on the following estimate on
$\mathcal{M}_{\lambda }^{\delta ,\gamma }f$.
For $0< \alpha<n$, the  fractional maximal operator $M_{\alpha}$ is defined by
$$(M_{\alpha}f)(x)=\sup_{Q}\frac{1}{|Q(x,l)|^{1-\alpha/n}}\int_{Q(x,l)}|f(y)|dy,$$
where the supremum is taken over all cubes $Q(x,l)$ in $\mathbb{R}^n$  whose center are at the origin
$x$, diameter $l(>0)$ and with the sides parallel to the axes. When $\alpha=0$, the operator $M_{0}$ is reduced to
the Hardy-Littlewood maximal operator, which is simply denoted by $M$.
\begin{thm}
\label{t20} Let $\delta <(n-1)/2$ and $1\leq p<\infty$. If $f\in I_{\lambda }(L^{p})(\mathbb{R}%
^{n})$, then for $\frac{n-1%
}{2}-\delta \leq \lambda \leq \gamma$, we have
\begin{equation*}
(\mathcal{M}_{\lambda }^{\delta ,\gamma }f)(x)\preceq ({M}_{\frac{n-1}{2}%
-\delta }g)(x)+I_{\frac{n-1}{2}%
-\delta}(|g|)(x)+({M}g)(x),
\end{equation*}%
where ${g}=:I_{-\lambda }f$,  $M_{(n-1)/2-\delta }$ is the fractional
Hardy-Littlewood maximal operator, and $I_{(n-1)/2-\delta }$ is the Riesz potential operator of
order $(n-1)/2-\delta$.
\end{thm}

We begin with some useful lemmas. Following the previous proof of lemma \ref{lem21}, we
can obtain the following lemma.
\begin{lem}
\label{lem31} For $\frac{n-1}{2}-\delta \leq \lambda\leq \gamma$, we have
\begin{equation*}
|K_{\lambda,1}^{\delta,\gamma}(x)|\preceq 1, \text{\
if\ } |x|\leq 1,
\end{equation*}
and
\begin{equation*}
|K_{\lambda,1}^{\delta,\gamma}(x)|\preceq |x|^{-\frac{%
n+1}{2}-\delta}, \text{\ if\ }|x|> 1.
\end{equation*}
\end{lem}

We also need the following lemmas in \cite{FZ}.

\begin{lem}[\protect\cite{FZ}]
\label{lem32} For $0\leq \lambda<\gamma$ and $\delta>-1$, we have
\begin{equation*}
|K_{\lambda,0}^{\delta,\gamma}(x)|\preceq \frac{1}{(1+|x|)^{n+\gamma-\lambda}%
},
\end{equation*}
and for $\lambda=\gamma$,
\begin{equation*}
|K_{\gamma,0}^{\delta,\gamma}(x)|\preceq \frac{1}{(1+|x|)^{L}},\ \text{for
any}\ L>0.
\end{equation*}
\end{lem}

\begin{lem}[\protect\cite{FZ}]
\label{lem33} For $\lambda> 0$ and $\delta>-1$, we have that
\begin{equation*}
|K_{\lambda,\infty}^{\delta,\gamma}(x)|\preceq \frac{1}{|x|^{n-\lambda}},\
\text{if}\ |x|<1
\end{equation*}
and if $|x|\geq 1$
\begin{equation*}
|K_{\lambda,\infty}^{\delta,\gamma}(x)|\preceq \frac{1}{|x|^{L}},\ \text{for
any}\ L>n.
\end{equation*}
\end{lem}

\medskip

\noindent\textit{Proof of Theorem \ref{t20}.} For $j=0,1,\infty$, define the
operators
\begin{equation*}
(\mathcal {T}_{R,\lambda ,j}^{\delta ,\gamma }g)^{\wedge }(\xi)=R^{\delta-{(n-1)}/{2}}
m_{\lambda ,j}^{\delta ,\gamma }(\xi/R)\widehat{g}(\xi),
\end{equation*}%
and their associated maximal operators
\begin{equation*}
(\mathcal{M}_{\lambda ,j}^{\delta ,\gamma }g)(x)=\sup_{R>1}|(T_{R,\lambda
,j}^{\delta ,\gamma }g)(x)|.
\end{equation*}
Denote by $(K_{\lambda ,j}^{\delta ,\gamma })_{1/R}(x)=R^{n}K_{\lambda
,j}^{\delta,\gamma }(Rx)$ the kernel of $\mathcal {T}_{R,\lambda ,j}^{\delta,\gamma }$,
where
\begin{equation*}
K_{\lambda ,j}^{\delta ,\gamma }(x)={m_{\lambda ,j}^{\delta ,\gamma }}^{\vee
}(x),\ j=0,1,\infty.
\end{equation*}

The sublinearity of the maximal function $\mathcal{M}_{\lambda }^{\delta
,\gamma }f$  leads to
\begin{equation*}
(\mathcal{M}_{\lambda }^{\delta ,\gamma }f)(x)\leq \sum_{j=0,1,\infty}(%
\mathcal{M}_{\lambda ,j}^{\delta ,\gamma }g)(x),
\end{equation*}%
where $\widehat{g}(\xi )=(I_{-\lambda }f)^{\wedge }(\xi )=|\xi |^{\lambda }%
\widehat{f}(\xi )$.

Consider $j=0$.  If $0\leq \lambda <\gamma$, using Lemma \ref{lem32}, then we have
\begin{align*}
& \quad \left\vert (\mathcal{T}_{R,\lambda ,0}^{\delta ,\gamma }g)(x)\right\vert
=\left\vert \int_{\mathbb{R}^{n}}R^{n+\delta-{(n-1)}/{2}}K_{\lambda
,0}^{\delta ,\gamma }(Ry)g(x-y)dy\right\vert \\
& \leq R^{n+\delta-{(n-1)}/{2}}\left(\int_{|Ry|<1}|K_{\lambda ,0}^{\delta ,\gamma
}(Ry)||g(x-y)|dy+\int_{|yR|\geq 1}|K_{\lambda
,0}^{\delta ,\gamma }(Ry)||g(x-y)|dy\right) \\
& \preceq R^{n+\delta-{(n-1)}/{2}}\left(\int_{|y|<1/R}|g(x-y)|dy+\sum_{k=0}^{\infty
}\int_{2^{k}/R\leq |y|<2^{k+1}/R}\frac{|g(x-y)|}{%
|Ry|^{n+\gamma -\lambda }}dy\right).
\end{align*}
Then by taking the supremum over $R>1$, we yield that
\begin{align*}
& \quad \sup_{R>1}\left\vert (\mathcal{T}_{R,\lambda ,0}^{\delta ,\gamma }g)(x)\right\vert\\
& \preceq (M_{\frac{n-1}{2}-\delta}g)(x)+\sum_{k=0}^{\infty
}2^{-k((n-1)/2-\delta+\gamma-\lambda)}(M_{\frac{n-1}{2}-\delta}g)(x) \\
& \preceq (M_{\frac{n-1}{2}-\delta}g)(x),
\end{align*} because of $(n-1)/2-\delta+\gamma-\lambda>0$.
The case $\lambda =\gamma$ is much simpler. We skip the details.

We next estimate the case $j=1$. Applying Lemma \ref{lem31}, we have
\begin{align*}
& \quad \left\vert (\mathcal{T}_{R,\lambda ,1}^{\delta ,\gamma }g)(x)\right\vert
=\left\vert \int_{\mathbb{R}^{n}}R^{n+\delta-{(n-1)}/{2}}K_{\lambda
,1}^{\delta ,\gamma }(Ry)g(x-y)dy\right\vert \\
& \preceq R^{n+\delta-{(n-1)}/{2}}\int_{|y|<1/R}|g(x-y)|dy+\int_{|y|\geq
1/R} \frac{R^{n+\delta-{(n-1)}/{2}}}{(|Ry|)^{{(n+1)}/2+\delta }}|g(x-y)|dy
\\
& \preceq R^{n+\delta-{(n-1)}/{2}}\int_{|y|<1/R}|g(x-y)|dy+\int_{|y|\geq
1/R}\frac{|g(x-y)|}{%
|y|^{{(n+1)}/2+\delta }}dy.
\end{align*}
Since $\delta<{(n-1)}/2$, we have
\begin{align*}
\sup_{R>1}\left\vert (\mathcal{T}_{R,\lambda ,1}^{\delta ,\gamma }g)(x)\right\vert
\preceq (Mg)(x)+I_{\frac{n-1}{2}-\delta}(|g|)(x).
\end{align*}

We now deal with the case $j=\infty$. It is easy to verify the case $\lambda
=\gamma $. If $0\leq \lambda <\gamma $, using Lemma \ref{lem33} by choosing $L=n$,
then we have
{\allowdisplaybreaks
\begin{align*}
& \quad \left\vert (\mathcal{T}_{R,\lambda ,\infty}^{\delta ,\gamma }g)(x)\right\vert
=\left\vert \int_{\mathbb{R}^{n}}R^{n+\delta-{(n-1)}/{2}}K_{\lambda
,\infty}^{\delta ,\gamma }(Ry)g(x-y)dy\right\vert \\
& \leq R^{n+\delta-{(n-1)}/{2}}\left(\int_{|Ry|<1}\frac{|g(x-y)|}{|Ry|^{n-\lambda }}%
dy + \int_{|yR|\geq 1}\frac{|g(x-y)|}{|Ry|^{n}}%
dy\right)\\
& \leq R^{n+\delta-{(n-1)}/{2}}\sum_{k=-\infty}^{0 }\int_{2^{k-1}\leq |Ry|<2^{k}}\frac{|g(x-y)|%
}{|Ry|^{n-\lambda }}dy \\
&\quad + R^{n+\delta-{(n-1)}/{2}}\sum_{k=0}^{\infty}\int_{2^{k}\leq |yR|< 2^{k+1}}%
\frac{|g(x-y)|}{|Ry|^{n}}dy.
\end{align*}}%
Then for $R>1$, we have
{\allowdisplaybreaks\begin{align*}
& \quad \sup_{R>1}\left\vert (\mathcal{T}_{R,\lambda ,\infty}^{\delta ,\gamma }g)(x)\right\vert\\
& \preceq \sum_{k=-\infty}^{0 }2^{k\lambda}\sup_{R>0}\left(\frac{R}{2^k}%
\right)^n\int_{|Ry|<2^{k}}|g(x-y)|dy \\
&\quad+ \sum_{k=0}^{\infty} 2^{k(\delta-(n-1)/2)}\sup_{R>0}\left(\frac{R}{2^{k+1}}%
\right)^{n-((n-1)/2-\delta)}\int_{|Ry|<2^{k+1}}|g(x-y)|dy \\
& \preceq (Mg)(x)+ (M_{\frac{n-1}{2}-\delta}g)(x),
\end{align*}}where we have used $\lambda\geq (n-1)/2-\delta>0$.
 Therefore, Theorem \ref{t20} is completely proved. \qed

\bigskip

\noindent\textit{Proof of Theorem \ref{t1}.} This proof is similar to
that of Theorem \ref{t3}. Suppose $f\in I_{\lambda}(L^p)(\mathbb{R}^n)$ for $%
1\leq p<\infty$. Then for any $\varepsilon>0$, we decompose $f$ into two
functions
\begin{equation*}
f(x)=g(x)+h(x),
\end{equation*}
such that $g\in \mathscr{S}(\mathbb{R}^{n})$ and $\|I_{-\lambda} h\|_{L^p(%
\mathbb{R}^n)}<\varepsilon$.

Since $g\in \mathscr{S}(\mathbb{R}^{n})$,  for any fixed $s>0,$ we conclude that
\begin{equation*}
\left\vert \left\{ x\in \mathbb{R}^{n}:\lim \sup_{R\rightarrow \infty
}\left\vert R^{\lambda+\delta-\frac{n-1}{2} }\left( (S_{R}^{\delta,\gamma
}g)(x)-g(x)\right) \right\vert >s /2\right\} \right\vert =0.
\end{equation*}%
Let $\widetilde{h}=I_{-\lambda}h$. Making use of the sublinearity of maximal
function, together with Theorem \ref{t20}, it follows that for any $s >0,$
{\allowdisplaybreaks \begin{align*}
& \quad \left\vert \left\{ x\in \mathbb{R}^{n}:\lim \sup_{R\rightarrow
\infty }\left\vert R^{\lambda+\delta-\frac{n-1}{2} }\left(
(S_{R}^{\delta,\gamma }f)(x)-f(x)\right) \right\vert >s\right\} \right\vert
\\
& \leq \left\vert \left\{ x\in \mathbb{R}^{n}:\left( \mathcal{M}_{\lambda }
^{\delta,\gamma }\widetilde{h}\right) (x)>s /2\right\} \right\vert \\
& \leq \left\vert \left\{ x\in \mathbb{R}^{n}:C\left( {M}_{\frac{n-1}{2}%
-\delta}\widetilde{h}\right) (x)>s /6\right\} \right\vert \\
&\quad +\left\vert \left\{ x\in \mathbb{R}^{n}:C {I}_{\frac{n-1}{2}%
-\delta}\left(\widetilde{|h|}\right) (x)>s /6\right\} \right\vert \\
& \quad+ \left\vert \left\{ x\in \mathbb{R}^{n}: C( {M} \widetilde{h})(x)
>s /6\right\} \right\vert \\
& \preceq \left(\frac{\|\widetilde{h}\|_{L^p(\mathbb{R}^n)}}{s}\right)^{1/q}
+\left(\frac{\|\widetilde{h}\|_{L^p(\mathbb{R}^n)}}{s}\right)^{1/p} \\
&\leq \left(s^{-1/q}+s^{-1/p}\right)\cdot\varepsilon^{1/q},
\end{align*}}where we have used that the boundednesses of the Riesz
potential operator $I_{\frac{n-1}{2}-\delta}$ (the
fractional maximal operator $M_{\frac{n-1}{2}-\delta}$)$: L^{p}(\mathbb{R}%
^n)\rightarrow L^{q,\infty}(\mathbb{R}^n)$ with $1/p-1/q=\frac{n-1-2\delta}{2n}$
 for $1\leq p<\frac{2n}{n-1-2\delta}$ and of the Hardy-Littlewood maximal operator
 $M: L^p(\mathbb{R}^n)\rightarrow L^{p,\infty}(%
\mathbb{R}^n)$ for $1\leq p<\infty$. This concludes the proof of the theorem.

\qed



\end{document}